\documentclass{amsart}
\usepackage[all]{xy}
\usepackage{amsmath}
\usepackage{amsfonts}
\usepackage{amssymb}
\usepackage{mathrsfs}
\usepackage{amscd}
\usepackage{soul,color}

\usepackage[notref,notcite]{showkeys}

\DeclareFontEncoding{OT2}{}{} 
\newcommand{\textcyr}[1]{%
 {\fontencoding{OT2}\fontfamily{wncyr}\fontseries{m}\fontshape{n}
 \selectfont #1}}
\newcommand{\Sha}{{\!\!\lbe\mbox{\textcyr{Sh}}}}

\theoremstyle{plain}
\newtheorem{theorem}{Theorem}
\newtheorem{mytheorem}{Theorem}[subsection]
\newtheorem{corollary}[mytheorem]{Corollary}
\newtheorem{lemma}[mytheorem]{Lemma}
\newtheorem{proposition}[mytheorem]{Proposition}
\newtheorem{definition}[mytheorem]{Definition}
\newtheorem{remark}[mytheorem]{Remark}

\def\le{\kern 0.03em}

\def\e{\kern 0.08em}
\def\be{\kern -.1em}
\def\lbe{\kern -.025em}

\def\F{{\mathbb F}}
\def\tF{{\mathbf F}}

\def\cG{{\mathcal G}}

\def\K{{\mathcal K}}
\def\L{{\mathcal L}}

\def\O{{\mathcal O}}

\def\Q{{\mathbb Q}}
\def\R{{\mathbb R}}

\def\W{{\mathcal W}}

\def\Z{{\mathbb Z}}

\DeclareMathOperator{\coh}{H}
\DeclareMathOperator{\Frob}{Frob}
\DeclareMathOperator{\Gal}{Gal}

\DeclareMathOperator{\res}{res}
\DeclareMathOperator{\rank}{rank}
\DeclareMathOperator{\corank}{corank}

\DeclareMathOperator{\Nm}{N}

\DeclareMathOperator{\image}{Im}
\DeclareMathOperator{\Sel}{Sel}

\DeclareMathOperator{\Hom}{Hom}
\DeclareMathOperator{\T}{T}
\DeclareMathOperator{\Tor}{Tor}
\DeclareMathOperator{\Tr}{Tr}

\DeclareMathOperator{\ord}{ord}
\DeclareMathOperator{\coker}{coker}

\begin{document}
\title[Selmer groups over $\Z_p^d$-extensions]{Selmer groups over
$\Z_p^d$-extensions}
\author{Ki-Seng Tan}
\address{Department of Mathematics\\
National Taiwan University\\%
Taipei 10764, Taiwan}
\email{tan@math.ntu.edu.tw}
\thanks{\textbf{Acknowledgement:} This research was supported in part by the National
Science Council of Taiwan, NSC97-2115-M-002-006-MY2, NSC99-2115-M-002-002-MY3.}
\begin{abstract}
Consider an abelian variety $A$ defined over a global field $K$ and let $L/K$ be a $\Z_p^d$-extension,
unramified outside a finite set of places of $K$,
with $\Gal(L/K)=\Gamma$. Let $\Lambda(\Gamma):=\Z_p[[\Gamma]]$ denote the Iwasawa algebra.
In this paper, we study how the characteristic ideal of the
$\Lambda(\Gamma)$-module $X_L$, the dual $p$-primary Selmer group, varies when $L/K$
is replaced by a intermediate $\Z_p^e$-extension.

\end{abstract}
\maketitle


\begin{section}{Main Results}\label{s:intro}

Let $A$ be a $g$-dimensional
abelian variety defined over a global field
$K$ and let $L/K$ be a $\Z_p^d$-extension, unramified outside a finite
set of places of $K$, with $\Gal(L/K)=\Gamma$.
For each finite intermediate extension $F/K$ of $L/K$, let $\Sel_{p^{\infty}}(A/F)$ denote the
$p$-primary Selmer group (see \S\ref{sub:ct})
and set
$$\Sel_{p^{\infty}}(A/L)=\varinjlim_{{F}}\Sel_{p^{\infty}}(A/F).$$
We endow  $\Sel_{p^{\infty}}(A/L)$ (resp. $\Sel_{p^{\infty}}(A/F))$ with the discrete topology and let $X_L$ (resp. $X_F$) denote its
Pontryagin dual group.
The main aim of this paper
is to study how the characteristic ideal of $X_L$ over $\Lambda(\Gamma):=\Z_p[[\Gamma]]$ ({\em{the Iwasawa algebra}})
varies, when $L/K$
is replaced by an intermediate $\Z_p^e$-extension $L'/K$.
Our result has many applications. In particular, it leads to a structure theorem of $Z_L$,
the Pontryagin dual of $\varinjlim_{{F}} \Q_p/\Z_p\otimes A(F)$ (see \S\ref{su:zero}).

\begin{subsection}{Notation}\label{su:ns}
Let $S$ denote the set of places of $K$ ramified over $L/K$.
For an algebraic extension $F/K$ and a place $w$ of $F$,
let $F_w$ denote the $w$-completion of $F $. If $w$ is a non-archimedean  place,
let $\O_{w}$, $m_{w}$ and $\F_w$ (or $\O_{F_w}$, $m_{F_w}$ and $\F_{F_w}$) denote the ring of integers, the maximal ideal and the residue field of $F_w$.
Also, denote $q_w=|\F_w|$.
We fix an algebraic closure ${\overline{K}}$ of $K$ and let ${K}^s\subset
{\overline{K}}$ denote the separable closure of $K$, and the same for $K_v$.

For an abelian group $D$, let $D_p$ (resp. $D_{div}$) denote the $p$-primary (resp. $p$-divisible)
part of $D_{tor}$, the torsion subgroup.
For a locally compact group $G$, let $G^{\vee}$ denote its Pontryagin dual group. In this paper, we always have
$G^{\vee}=\Hom_{cont}(G,\Q_p/\Z_p)$ as $G$ will be either pro-$p$ or $p$-primary.
If $\O$ is the ring of integers of a finite extension $\mathcal{Q}$ of $\Q_p$ and $G$ is an $\O$-module,
we endow $G^\vee$ with the $\O$-module structure by setting $a\cdot \varphi(g)=\varphi(a\cdot g)$, $a\in\O$, $\varphi\in G^\vee$, $g\in G$.
As $\O$-modules, $G$ is cofinitely generated if and only if $G^{\vee}$ is finitely generated,
and denote $\corank_{\O}(G):=\rank_{\O}(G^\vee)$.

If $G$ is a $\Z_p$-module, write $\O G$ for $\O\otimes_{\Z_p}G$. Then we can identify $(\O G)^\vee$ with $\O G^\vee$ by introducing a non-degenerate
pairing $[\;,\;]: \O G\times \O G^\vee\longrightarrow \Q_p/\Z_p$ as follow.
First choose a generator $\delta\in\O$ of the {\em{different}} of the filed extension $\mathcal{Q}/\Q_p$
and set $\Tr^*(x)=\Tr_{\mathcal{Q}/\Q_p}(\delta^{-1}\cdot x)$ for $x\in \mathcal{Q}$. 
If $a\in \mathcal{Q}/\O$ is the residue class of some $y\in \mathcal{Q}$ modulo $\O$, let
$\T^*(a)\in\Q_p/\Z_p$ denote the residue class of $\Tr^*(y)$ modulo $\Z_p$.
Then $\mathrm{Q}:\O\times \mathcal{Q}/\O\longrightarrow \Q_p/\Z_p$ given by $\mathrm{Q}(x,a):=\T^*(xa)$
is a non-degenerate pairing. Let $<\;,\;>:\O G\times \O G^\vee\longrightarrow \mathcal{Q}/\O$ be the $\O$-pairing given by $<g,\phi>=\phi(g)$,
for $g\in G$, $\phi\in G^\vee$. Then define $[\alpha,\beta]=\T^*(<\alpha,\beta>)$. Let $e_1,...,e_m$ be a $\Z_p$-basis of $\O$.
If $\beta=\sum_i e_i\otimes \phi_i$, $\phi_i\in G^\vee$, satisfies $[\O G,\beta]=0$, then for all $x\in\O$, $g\in G$,
$$0=[x\otimes g, \beta]= \T^*(x\cdot \sum_ie_i\otimes <g,\phi_i>)=\mathrm{Q}(x,\sum_i e_i\otimes <g,\phi_i>).$$
Since $\mathrm{Q}$ is non-degenerate, $\sum_i e_i\otimes <g,\phi_i>=0$, and hence $<g,\phi_i>=0$ for all $g$.
Consequently, $\phi_i=0$ for every $i$, whence $\beta=0$. Similarly, if $\alpha\in\O G$ satisfies $[\alpha,\O G^\vee]=0$, then $\alpha=0$.

If $G$ is a $\Gamma$-module, let $\Gamma$ acts on $G^{\vee}$ by
${}^{\gamma}\varphi(g):=\varphi({}^{\gamma^{-1}}g)$.
The identification $(\O G)^\vee=\O G^\vee$ depends on the choice of $\delta$. However, it
alters neither the $\O$-module structure nor the $\Gamma$-module structure (if exists) on $(\O G)^\vee$.

Let $\mu_{p^m}$ denote the $p^m$th root of unity and write $\mu_{p^{\infty}}=\bigcup_m \mu_{p^m}$ regarded as a discrete subgroup of ${\overline{\Q}}_p^{\times}$.
Let $\widehat{\Gamma}$ denote the group of all continuous characters from $\Gamma$ to $\mu_{p^{\infty}}$ and
let $\Gal({\overline{\Q}}_p/\Q_p)$ act on it  via the action on $\mu_{p^{\infty}}$.
Thus, $\widehat{\Gamma}=\Gamma^{\vee}$ as topological groups, while $\Gal({\overline{\Q}}_p/\Q_p)$ acts non-trivially
on $\widehat{\Gamma}$ but trivially on $\Gamma^\vee$.
If $\omega\in \widehat{\Gamma}$ with the image $\image(\omega)=\mu_{p^m}$,
write
$\O_{\omega}=\Z_p[\mu_{p^m}]\subset {\overline \Q}_p$. If $\O$ contains $\O_{\omega}$ and $G$ is an $\O$-module
with a continuous action of $\Gamma$, write (for the $\omega$-eigenspace)
$$G^{(\omega)}:=\{g\in G\;\mid\; {}^{\gamma}g=\omega(\gamma)\cdot g\}.$$

For a finitely generated $\Lambda(\Gamma)$-module $W$,
let $\chi_{\Lambda(\Gamma)}(W)$ denote its characteristic ideal (see \S\ref{su:char}).
Denote $\Gamma'=\Gal(L'/K)$, $\Lambda(\Gamma')=\Z_p[[\Gamma']]$.
Our result also covers the $d=1$ case in which
$\Gamma'=0$, $\Lambda(\Gamma')=\Z_p$, and we define $\chi_{\Lambda(\Gamma')}(W)=\chi_{\Z_p}(W)$,
the usual characteristic ideal of $\Z_p$-module.

Let $A[p^m]$ denote the kernel of the multiplication by $p^m$ on $A$ viewed as a sheaf on the flat topology of $K$
and denote $A[p^{\infty}]=\bigcup_m A[p^m]$. In particular, $A[p^{\infty}](K)=A(K)_p$.
Let $A^t$ denote the dual abelian variety.

\end{subsection}
\begin{subsection}{A local condition}\label{su:localcondition}
It is well known that if $K$ is a number field, then the following question has an affirmative
answer (see below).
$$Is \;X_L\; finitely\; generated\;
over\;\Lambda(\Gamma)?$$
In general, the answer could be obtained via the  following
local criterion.
\begin{proposition}\label{p:iwasawa}
The Iwasawa module $X_L$ is finitely generated over $\Lambda(\Gamma)$
if and only if
at each place $v\in S$, the local cohomology group
$\coh^1(\Gamma_v,A(L_v))$ is cofinitely generated over $\Z_p$.
\end{proposition}
The proof, based on results in \cite{tan10}, is given in \S\ref{su:tiw}.
The condition of the proposition holds if $K$ is a number field (Corollary \ref{c:nf}),
or 
if at every ramified place,
$A$ has either good ordinary reduction or split-multiplicative
reduction \cite[Theorem 5]{tan10}.
However, if $char.(K)=p$ and the reduction of $A$ at a
place $v\in S$ is an abelian variety without non-trivial $p$-torsion points, then the condition fails to hold (Theorem \ref{t:ss}).

\end{subsection}
\begin{subsection}{The specialization data}\label{su:desda}
For the rest of this paper except \S\ref{su:deep}, we shall
assume that {\em{every $v\in S$ is either good ordinary or
split-multiplicative}}, and hence $X_L$ is finitely generated over $\Lambda(\Gamma)$.
Also, for simplicity, we assume that $char.(K)=p$, if $K$
is not a number field. Write $\Theta_L=\chi_{\Lambda(\Gamma)}(X_L)$.
Extend the canonical map $\Gamma\longrightarrow\Gamma'$ 
to the continuous $\Z_p$-algebra homomorphism (the specialization map) $p_{L/L'}:\Lambda(\Gamma)\longrightarrow \Lambda(\Gamma')$.
Then the following question arises:
$$What\; is\; the\; relation\; between \; p_{L/L'}(\Theta_L) \;and\;
\Theta_{L'}?$$
To illustrate our answer, some simplification and notation are in order.
First, by choosing a sequence
$L'\subset L''\subset \cdots \subset L^{(i)}\subset \cdots\subset L^{(d-e)}\subset L$
with each $\Gal(L^{(i)}/K)\simeq \Z_p^{e-1+i}$, we can write
$p_{L/L'}=p_{L''/L'}\circ\cdots\circ p_{L^{(i+1)}/L^{(i)}}\circ\cdots\circ p_{L/L^{(d-e)}}$,
and hence answer the question for $p_{L/L'}$ by answering that for every $p_{L^{(i+1)}/L^{(i)}}$.
Therefore, {\em{without loss of generality, we may assume that}} $e=d-1$. We shall make such assumption and
then fix a topological generator $\psi$ of $\Psi:=\Gal(L/L')$.

\subsubsection{The global factor}
Let $K'/K$ be a $\Z_p$-extension and let $\sigma$ be a topological generator of the Galois group.
If $\epsilon_1,...,\epsilon_l$ are eigenvalues, counted with multiplicities, of the action of $\sigma$ on
the Tate module $\mathrm{T}_p A[p^{\infty}](K')$. Then the product
$\prod_{j=1}^l (1-\epsilon_j^{-1}\sigma)\subset \Z_p[[\Gal(K'/K)]]$ is nothing but the characteristic ideal of
$\mathrm{T}_p A[p^{\infty}](K')$ over $\Z_p[[\Gal(K'/K)]]$, and in particular, the ideal
$$\mathrm{w}_{K'/K}:=\prod_{j=1}^l (1-\epsilon_j^{-1}\sigma) (1-\epsilon_j^{-1}\sigma^{-1})$$
is independent of the choice of $\sigma$ (See Proposition \ref{p:torcomp1}). If $K'/K$ is a $\Z_p^e$-extension with $e\geq 2$, set
$\mathrm{w}_{K'/K}=(1)$.
\begin{definition}\label{d:glfactor}
Define the $\Lambda(\Gamma')$-ideal
$$\varrho_{L/L'}=
\begin{cases}
\mathrm{w} _{L'/K} & \text{if}\; d\geq 2;\\
\frac{|A[p^{\infty}](K) |^2 } {| A[p^{\infty}](K)\bigcap A[p^{\infty}](L)_{div}|^2}
&\; \text{if}\; d=1.\\
\end{cases}
$$
\end{definition}

\subsubsection{Local factors at unramified places}
\begin{definition}\label{d:badunram}
For each $v$, let ${\Pi_v}$ denote the group of the connected components of the
closed fiber {\em{(over ${\overline{ \F}}_v$)}}
of the N$\acute{\text{e}}$ron model of $A/K_v$ and let $\pi_v$ denote the $\Z_p$-ideal $(|\Pi_v^{\Gal({\overline{\F}}_v/\F_v)}|)$.
\end{definition}

\subsubsection{Local factors at good ordinary places}
Suppose that $A$ has good ordinary reduction $\bar A$ at $v$.
Then eigenvalues of the Frobenius endomorphism
$\tF_v:{\bar A}\longrightarrow {\bar A}$ over $\F_v$ are, counted with multiplicities,
\begin{equation*}
\alpha_1,...,\alpha_g,q_v/\alpha_1,...,q_v/\alpha_g,
\end{equation*}
where $\alpha_1,...,\alpha_g$ are eigenvalues of the (twist) matrix $\mathbf{u}$ of the action on the Tate module
of ${\bar A}[p^{\infty}]$ by
the Frobenius substitution $\Frob_v\in \Gal({\overline{ \F}}_v/\F_v)$
(\cite[Corollary 4.37]{maz}). 

\begin{definition}\label{d:goodord}
Suppose  $A$ has good ordinary reduction $\bar A$ at $v$
and $L'/K$ is unramified at $v$ with the Frobenius element
$[v]_{L'/K}\in\Gamma'$. Define
$$\textsf{f}_{L',v}
:=\prod_{i=1}^g (1-\alpha_i^{-1}\cdot [v]_{L'/K})\times
\prod_{i=1}^g (1-\alpha_i^{-1}\cdot [v]_{L'/K}^{-1})\subset \Lambda(\Gamma').$$
\end{definition}

\subsubsection{Local factors at split multiplicative places}
Suppose $A$ has split multiplicative reduction at $v$. This means
there is a rank $g$ lattice $\Omega_v\simeq \Z\times\dots\times\Z$ sitting
inside the torus $T=(K_v^{\times})^g$ so that
$T/\Omega_v$ is isomorphic to the rigid analytic space associated to $A$
(see \cite{ger72}). In particular,
\begin{equation}\label{e:desplit}
A({\overline{K}}_v)\simeq ({\overline{K}}_v^{\times})^g/\Omega_v.
\end{equation}

Consider the composition $\Omega_v\longrightarrow (K_v^{\times})^g\stackrel{{\textsf{R}}_v^g}{\longrightarrow} (\Gamma_v)^g$
where $\textsf{R}_v:K^{\times}_v\longrightarrow \Gamma_v$ is the local reciprocity map, and extend it $\Z_p$-linearly to

\begin{equation}\label{e:mcrv}
\xymatrix{\mathcal{R}_v:\Z_p\otimes_{\Z}\Omega_v \ar[r] & (\Gamma_v)^g.}
\end{equation}

\begin{definition}\label{d:splitmul}
Define $\mathfrak{w}_v=\chi_{\Z_p}(\coker\be\left[\e \mathcal{R}_v\right])$.
\end{definition}
\end{subsection}
\begin{subsection}{The main theorem}\label{su:maint}
Here is our main theorem. Recall that $\Psi=\Gal(L/L')$.
\begin{theorem}\label{t:compatible} Suppose $d\geq 1$ and assume the above
notation. Then we have
$$\Theta_{L'} \cdot \vartheta_{L/L'}
=\varrho_{L/L'}\cdot
p_{L/L'}(\Theta_L),$$
where $\vartheta_{L/L'}:=\prod_v\vartheta_v$ with each
$\vartheta_v$ an ideal of $\Lambda(\Gamma')$ defined by the following conditions:
\begin{enumerate}
\item[(a)] Suppose $v\not\in S$. If $\Psi_v\not=0$,
then $\vartheta_v=\pi_v$; otherwise, $\vartheta_v=(1)$.
\item[(b)] Suppose $v\in S$ and $A$ has good ordinary reduction at $v$.
If $v$ is unramified over $L'/K$, then $\vartheta_v=\textsf{f}_{L',v}$;
otherwise  $\vartheta_v=(1)$.
\item[(c)] Suppose $v\in S$ and $A$ has split-multiplicative reduction at $v$.
Then
$$
\vartheta_v=\begin{cases}
\Lambda(\Gamma')\cdot \mathfrak{w}_v, & \text{if}\ \  \Gamma'_v=0;\\
(\sigma-1)^{g}, & \text{if}\ \ \Psi_v\simeq\Z_p  \;\text{and}\;\; \Gamma'_v \;\text{is topologically generated by}\; \sigma
;\\
(1),& \text{otherwise.}
\end{cases}
$$
\end{enumerate}

\end{theorem}
The proof will be completed in \S\ref{su:pft}.
The tools, local and global, for the proof will be established in \S2, \S3, and \S4.
See \S\ref{su:imc} for the application of the theorem to the Iwasawa Main Conjecture.
Here is an immediate application.

\begin{theorem}\label{t:otr09}
Suppose $char.(K)=p$ and $L$ contains the constant $\Z_p$-extension of $K$.
Then $X_{L}$ is torsion over $\Lambda(\Gamma)$.
\end{theorem}
See \cite{maru07} for examples of non-torsion $X_L$ in the number field case, while examples in characteristic $p$ can be found
in \cite[ Appendix]{lltt13}.

\begin{proof}
Let $\F_q$ denote the constant field of $K$ and let $L_0=K\F_{q^{p^{\infty}}}$ be the constant $\Z_p$-extension over $K$
with $\Gamma^0=\Gal(L_0/K)$ topologically generated by the Frobenius substitution $\Frob_q:x\mapsto x^q$.
The theorem is already proved in \cite{otr09} for the $L=L_0$ case. This means $\Theta_{L_0}\not=0$.
By repeatedly applying Theorem \ref{t:compatible} ($d-1$ times), we deduce
$$p_{L/L_0}(\Theta_L)\cdot \mathrm{w}_{L_0/K}
=\Theta_{L_0}\cdot \prod_{v}\vartheta_v.$$
Since $\Gamma^0_v\not=0$, for all $v$, the factor $\vartheta_v$ equals one of $(1)$, $(\Frob^{\deg(v)}-1)^g$, or $\mathsf{f}_{L_0,v}$.
In particular, $\vartheta_v\not=0$, for all $v$. 
Therefore, $p_{L/L_0}(\Theta_L)\not=0$, and hence $\Theta_L\not=0$.
\end{proof}
\end{subsection}
\begin{subsection}{The Iwasawa main conjecture}\label{su:imc}
Possibly, Theorem \ref{t:compatible} could be useful for determining an explicit generator of $\chi_{\Lambda(\Gamma)}(X_L)$.
Assume that an explicitly given element $\theta_L'\in \Lambda(\Gamma)$ is already known to be a generator of the characteristic ideal of a submodule $X'_L$ of $X_L$, and we want to see if actually
\begin{equation}\label{e:actually}
\Theta_L=(\theta_L').
\end{equation}
In addition, assume that there exists an intermediate $\Z_p^e$-extension $L_0$ of $L/K$
such that $\Theta_{L_0}$, the characteristic ideal of $X_{L_0}$ over $\Lambda(\Gal(L_0/K))$, is explicitly given.
Then by applying Theorem \ref{t:compatible}, we can obtain an explicit expression of $p_{L/L_0}(\Theta_L)$ in terms of $\Theta_{L_0}$
and other factors. Thus, by checking the explicit expressions, we would be able to determine if
\begin{equation}\label{e:able}
p_{L/L_0}(\Theta_L)=(p_{L/L_0}(\theta_L')).
\end{equation}
The point is that Equations \eqref{e:actually} and \eqref{e:able} are indeed equivalent. To see this, we only need to write
$$\Theta_L=(\theta_L'\cdot \theta''),\;\text{for some}\; \theta''\in \Lambda(\Gamma) $$
and observe that $\theta''$ is a unit of $\Lambda(\Gamma)$
if and only if its image $p_{L/L_0}(\theta'')$ is a unit of $\Lambda(\Gal(L_0/K))$.
In the function field case, $L_0$ could be taken to be the constant $\Z_p$-extension, since an explicit expression of
$\Theta_{L_0}$ is already given in \cite{lltt13} (for semi-stable $A$). We can also apply the theorem in the reverse direction: if
\eqref{e:actually} is already known then we can use the theorem together with \eqref{e:able} to determine an explicit expression of
$\Theta_{L_0}$. In \cite{lltt13} this method is used in the case where $char.(K)=p$ and $A$ is a constant ordinary abelian variety.

\end{subsection}

\begin{subsection}{The zero set of $\Theta_L$ and the structure of $X_L^0$}\label{su:zero} Our theory is useful for determining the  $\Lambda(\Gamma)$-modules structures of
$$Y_L:=(\varinjlim_{{F}}\Sel_{p^{\infty}}(A/F)_{div})^{\vee},$$
$$Z_L:=(\varinjlim_{{F}} (\Q_p/\Z_p)\otimes A(F))^\vee,$$
and
$$
X_L^0:=(\le\bigcup_{F}\e(\Sel_{p^{\infty}}(A/L)^{\Gal(L/F)}\e)_{div}\le)^{\vee}$$
as well. In general, we have the surjections $\xymatrix{X_L^0 \ar@{->>}[r] &  Y_L \ar@{->>}[r] & Z_L}$ due to the maps
$$\xymatrix{(\Q_p/\Z_p)\otimes A(F) \ar@{^{(}->}[r] & \Sel_{p^{\infty}}(A/F)_{div} \ar[r]^-{res_{F}} & (\Sel_{p^\infty}(A/L)^{\Gal(L/F)})_{div}}.$$
The above inclusion is from the Kummer exact sequence, it is an isomorphism if the $p$-primary part of the Tate-Shafarevich group of $A$ over $F$ is finite. Thus, if this holds for all $F$ then $Z_L=Y_L$.
In contrast, by the {\em{control theorem}} (see e.g. \cite[Theorem 4]{tan10}) if $L/K$ only ramifies at good ordinary places then the restriction map $res_F$ is surjective for every $F$, and hence $Y_L=X_L^0$.

\subsubsection{The zero set}
The structures of $X_L^0$, $Y_L$ and $Z_L$ are
related to the zero set of $\Theta_L$. For $\theta\in\Lambda(\Gamma)$ define the zero set
$$\Delta_{\theta}:=\{\omega\in {\widehat{\Gamma}}\;\mid\; p_{\omega}(\theta)=0\},$$
where $
p_{\omega}:\O_{\omega}\Lambda(\Gamma)  \longrightarrow  \O_{\omega}
$ is the $\O_{\omega}$-algebra homomorphism extending $\omega:\Gamma\longrightarrow \O_{\omega}^\times$.
Note that for each $\omega\in \widehat{\Gamma}$, the eigenspace $(\O_{\omega}\Sel_{p^{\infty}}(A/L))^{(\omega)}$ is cofinitely generated over $\O_{\omega}$ as it is the Pontryagin dual of the finitely generated $\O_\omega$-module $\O X_L/\ker\le\left [ \le p_\omega \le\right]\cdot X_L$.

\begin{definition}\label{d:rs}
For each $\omega\in \widehat{\Gamma}$, denote
$s(\omega):=\corank_{\O_{\omega}}  \e  (\O_{\omega}\Sel_{p^{\infty}}(A/L))^{(\omega)}$.
\end{definition}
We have the inclusions
$$((\O_\omega \Sel_{p^{\infty}}(A/L))^{(\omega)})_{div} \subset  ((\O_\omega \Sel_{p^{\infty}}(A/L)^{\ker\le\left[\omega\right]})_{div})^{(\omega)}\subset (\O_\omega\le\bigcup_{F}\e(\Sel_{p^{\infty}}(A/L)^{\Gal(L/F)}\e)_{div}\le)^{(\omega)},$$
where the left term is just the $p$-divisible part of the term. Hence,
\begin{equation}\label{e:news}
s(\omega)=\corank_{\O_{\omega}}  \e  (\O_\omega\le\bigcup_{F}\e(\Sel_{p^{\infty}}(A/L)^{\Gal(L/F)}\e)_{div}\le)^{(\omega)}.
\end{equation}

\begin{theorem}\label{t:root}
A character $\omega\in \widehat{\Gamma}$ is contained in
$\triangle_{\Theta_L}$
if and only if
$s(\omega)>0$.
\end{theorem}

This theorem is proved in \S\ref{su:root}.
Let $\theta\in \Lambda(\Gamma)$ be an element vanishing on $\Delta_{\Theta_L}$. By this, we mean that $p_{\omega}(\theta)=0$
for every $\omega\in \Delta_{\Theta_L}$. Since the $\O_{\omega}\Lambda(\Gamma)$-structure of $(\O_{\omega}\Sel_{p^{\infty}}(A/L))^{(\omega)}$
factors through $\xymatrix{\O_{\omega}\Lambda(\Gamma) \ar[r]^-{p_{\omega}} & \O_{\omega}}$, we must have $\theta\cdot (\O_{\omega}\Sel_{p^{\infty}}(A/L))^{(\omega)}=p_{\omega}(\theta)\cdot (\O_{\omega}\Sel_{p^{\infty}}(A/L))^{(\omega)}=0$ for every $\omega\in \Delta_{\Theta_L}$. On the other hand, Theorem \ref{t:root} says if $\omega\not\in \Delta_{\Theta_L}$
then $(\O_{\omega}\Sel_{p^{\infty}}(A/L))^{(\omega)}$ is finite. Thus,
$\theta\cdot(\O_{\omega}\Sel_{p^{\infty}}(A/L))^{(\omega)}$ is always finite for all $\omega\in \widehat{\Gamma}$.

For each finite intermediate extension $F$ of $L/K$,
denote $\Gamma(F):=\Gal(F/K)$ and choose $\O$ so that it contains $\O_{\omega}$ for every
$\omega\in \widehat{\Gamma}(F):=\Hom(\Gamma(F),\mu_{p^{\infty}})$ regard as a finite subgroup of
$\widehat{\Gamma}$. Consider the elements $e_{\omega}:=\sum_{\gamma\in \Gamma(F)}\omega(g)^{-1}\cdot g\in \O[\Gamma(F)]$,
$\omega\in \widehat{\Gamma}(F)$, which are ${|\Gamma(F)|}$-multiples of idempotents.
Multiplying any finite $\O[\Gamma(F)]$-module $W$ by $e_\omega$'s, we can form a homomorphism
$$\bigoplus_{\omega\in \widehat{\Gamma}(F)} W^{(\omega)}\longrightarrow W$$
of finite kernel and cokernel. In particular, by taking $W=((\O\Sel_{p^{\infty}}(A/L))^{\Gal(L/F)})^\vee$ and by the duality,
we have a homomorphism
\begin{equation}\label{e:amorphism}
\O\Sel_{p^{\infty}}(A/L)^{\Gal(L/F)}\longrightarrow \bigoplus_{\omega\in \widehat{\Gamma}(F)} (\O\Sel_{p^{\infty}}(A/L))^{(\omega)}
\end{equation}
of finite kernel and cokernel. Then by multiplying both sides of \eqref{e:amorphism} by $\theta$ we see that
\begin{equation}\label{e:thetazero}
\theta\cdot (\Sel_{p^{\infty}}(A/L)^{\Gal(L/F)})_{div}=0
\end{equation}
as the left-hand side of the equality is finite and $p$-divisible.
By \cite[Proposition 3.3]{grn03} (see \cite[Corollary 3.2.4]{tan10} and the discussion in \S3.3 of the paper for the characteristic $p$ case), if
$$\res_{L/F}:\coh^1(F, A[p^{\infty}])\longrightarrow \coh^1(L,A[p^{\infty}])^{\Gal(L/F)}$$
denote the restriction map, then
\begin{equation}\label{e:kerreslf}
|  \ker\be\left[\e\res_{L/F}\le\right] |<\infty,
\end{equation}
and
\begin{equation}\label{e:cokerreslf}
|  \coker\be\left[\e\res_{L/F}\le\right] |<\infty.
\end{equation}
Then \eqref{e:thetazero} and \eqref{e:kerreslf} imply $\theta\cdot\Sel_{p^{\infty}}(A/F)_{div}=0$ as it is also finite and $p$-divisible.
We have proved:
\begin{corollary}\label{c:vanishingtheta}
If $\theta\in\Lambda(\Gamma)$ vanishes on $\Delta_{\Theta_L}$, then $\theta$ annihilates $(\Sel_{p^{\infty}}(A/L)^{\Gal(L/F)})_{div}$,
$\Sel_{p^{\infty}}(A/F)_{div}$ and $(\Q_p/\Z_p)\otimes A(F)$, for all $F$. Hence $\theta\cdot X_L^{0}=\theta\cdot Y_L=\theta\cdot Z_L=0$.
\end{corollary}

\subsubsection{A theorem of Monsky}
Now we recall a theorem of Monsky (\cite[Lemma 1.5 and Theorem 2.6]{mon81}).
A subset $T\subset \widehat{\Gamma}$ is called a $\Z_p$-flat of codimension $k>0$, if
there exist $\gamma_1,...,\gamma_k\in \Gamma$ expandable to a $\Z_p$-basis of $\Gamma$
and $\zeta_1,...,\zeta_k\in\mu_{p^{\infty}}$ so that
\begin{equation*}
T=T_{\gamma_1,...,\gamma_k;\zeta_1,...,\zeta_k}:=\{\omega \in \widehat{\Gamma}\;\mid\; {\omega}(\gamma_i)=\zeta_i,i=1,...,k\}.
\end{equation*}

\begin{theorem} {\em{(Monsky)}}
If $\theta\in\Lambda(\Gamma)$ is non-zero, then $\Delta_{\theta}\not=\widehat{\Gamma}$ and
is a finite union of $\Z_p$-flats.
\end{theorem}

Note that for a given $\theta\in\Lambda(\Gamma)$,
if $T\subset \Delta_\theta$ then ${}^\sigma T\subset\Delta_\theta$ for all $\sigma\in \Gal({\overline{\Q}}_p/\Q_p)$,
as $\Delta_\theta$ is invariant under the action of the Galois group.
Also, if $T_{\gamma,\zeta}\subset \Delta_{\theta}$ with $\zeta\in\O$ then $\gamma-\zeta$ divides $\theta$ in $\O\Lambda(\Gamma)$, and vice versa
(see \cite[Lemma 3.3.3]{lltt13} and its proof). In this case, $\gamma-{}^\sigma\zeta$ also divides $\theta$.
\begin{definition}\label{d:simple}
An element $f\in\Lambda(\Gamma)$ is simple, if there exist $\gamma\in\Gamma- \Gamma^p$ and
$\zeta\in\mu_{p^{\infty}}$ so that
$$f=f_{\gamma,\zeta}:=\prod_{\sigma\in\Gal(\Q_p(\zeta)/\Q_p)} (\gamma-{}^{\sigma}\zeta).$$
\end{definition}

If $\zeta$ is of order $p^{n+1}$ and $t_i=\gamma_i-1$, $i=1,...,d,$ where $\gamma_1,...,\gamma_d$ is a $\Z_p$-basis of $\Gamma$,
then $f_{\gamma_1,\zeta}$ is nothing but the polynomial $\sum_{i=0}^{p-1}(t_1+1)^{ip^{n}}$ that is irreducible
in $\Z_p[t_1]$. Hence, a simple element is irreducible in $\Lambda(\Gamma)=\Z_p[[t_1,...,t_d]]$.
Obviously, $\Delta_{f_{\gamma,\zeta}}=\bigcup_{\sigma} {}^{\sigma}T_{\gamma,\zeta}$.
In particular, two simple elements $f$ and $g$ divide each other if and only if $\Delta_f=\Delta_g$. On the other hand, if $T=T_{\gamma_1,...,\gamma_k;\zeta_1,...,\zeta_k}$,
$k\geq 2$, then we can find two relatively prime simple elements both vanishing on $T$, for example,
$f_{\gamma_1,\zeta_1}$ and $f_{\gamma_2,\zeta_2}$.

\subsubsection{The structure of $X_L^0$}
If $W$ is a torsion $\Lambda(\Gamma)$-module then $\theta:=\chi_{\Lambda(\Gamma)}(W)\not=0$ and there exists
an pseudo-isomorphism
\begin{equation}\label{e:Msim}
\iota: (\Lambda(\Gamma)/(f_1^{b_1}))^{a_1}\oplus \cdots \oplus (\Lambda(\Gamma)/(f_l^{b_l}))^{a_l}\oplus \bigoplus_{j=1}^m\Lambda(\Gamma)/(\xi_i) \longrightarrow W.
\end{equation}
where $a_1,...,a_l$, $b_1,...,b_l$ are positive integers, $f_1,...,f_l$ are all the relatively prime simple factors of $\theta$,
and $\xi_1,...,\xi_m\in\Lambda(\Gamma)$ are not divided by any simple element ($l=0$ or $m=0$ is allowed).
The product $\phi=f_1\cdot\cdots\cdot f_l$
vanishes on every codimension one $\Z_p$-flat of $\Delta_{\theta}$. By the above argument,
we can find two products $\varepsilon=g_1\cdot\cdots\cdot g_m$ and $\varepsilon'=g_1'\cdot\cdots\cdot g'_{m'}$,
relatively prime to $\phi$ and to each other,
of simple elements so that both $\varepsilon$ and $\varepsilon'$ vanish on every $\Z_p$-flat of $\Delta_{\theta}$ of codimension
grater than $1$. Then both $\phi\varepsilon$ and $\phi\varepsilon'$ vanish on $\Delta_{\theta}$.
Note that $\iota$ is actually injective as its domain of definition contains no non-trivial pseudo-null submodule (see \S\ref{su:char}).

If $X_L$ is torsion then by taking $W=X_L$ in \eqref{e:Msim} we obtain the exact sequence
\begin{equation}\label{e:xlsim}
\xymatrix{0\ar[r] & \bigoplus_{i=1}^l (\Lambda(\Gamma)/(f_i^{b_i}))^{a_i}\oplus \bigoplus_{j=1}^m\Lambda(\Gamma)/(\xi_i) \ar[r] & X_L \ar[r] & N\ar[r] & 0},
\end{equation}
for some pseudo-null $N$. Let $\phi\varepsilon$ and $\phi\varepsilon'$ be as above.
Let $\sim$ denote pseudo-isomorphism.

\begin{theorem}\label{t:x0}
Suppose $X_L$ is torsion over $\Lambda(\Gamma)$ and assume the above notation. Then both $\phi\varepsilon$ and $\phi\varepsilon'$
annihilate  $X_L^{0}$, $Y_L$, $Z_L$ and $(\Sel_{p^{\infty}}(A/L)^{\Gal(L/F)})_{div}$, $\Sel_{p^{\infty}}(A/F)_{div}$, $(\Q_p/\Z_p)\otimes A(F)$ for all finite intermediate extension $F$ of $L/K$.
Moreover, $X_L^0$ is pseudo isomorphic to $X_L/\phi\cdot X_L$. Namely, if $a_1,...,a_l$ are as in {\em{\eqref{e:xlsim}}}, then
$$X_L^0\sim (\Lambda(\Gamma)/(f_1))^{a_1}\oplus \cdots \oplus (\Lambda(\Gamma)/(f_l))^{a_l}.$$
\end{theorem}

\begin{proof}
The first assertion follows from Corollary \ref{c:vanishingtheta}. Consequently, $\phi\cdot X_L^0$ is pseudo-null,
being annihilated by relatively prime $\varepsilon$ and $\varepsilon'$. Thus, $X_L^0\sim X_L^0/\phi\cdot X_L^0$. By taking $W=X_L^0$
in \eqref{e:Msim} we obtain the exact sequence
$$\xymatrix{0\ar[r] & \bigoplus_{i=1}^l(\Lambda(\Gamma)/(f_i))^{c_i}\ar[r]^-{\iota} &  X_L^0\ar[r] & M\ar[r] & 0},$$
for some non-negative integers $c_1,...,c_l$ and some pseudo-null $M$. By comparing this exact sequence with \eqref{e:xlsim}
using the fact that $\Lambda(\Gamma)/(\phi, \xi_i)$ is pseudo-null, we see that $X_L/\phi X_L\sim X_L^0$ if and only if $c_i=a_i$ for each $i$.
We shall only show $c_1=a_1$, as the rest can be proved in a similar way.

First choose a $\xi\in \Lambda(\Gamma)$ that annihilates $M$ and is relatively prime to $\phi$. Let $E_\omega$ denote the field $\O_\omega\Q_p$ and via $\xymatrix{\Lambda(\Gamma) \ar[r]^-{p_\omega} &\O_\omega}\subset E_\omega$ we consider the map $\iota_\omega:=E_{\omega}\otimes_{\Lambda(\Gamma)}\iota$ for an
$$\omega\in \Delta_{f_1}\le -\le (\Delta_{f_2}\cup\cdots\cup
\Delta_{f_l}\cup\Delta_{\xi}\cup \Delta_{\xi_1} \cup \cdots \cup\Delta_{\xi_m}).$$
Now the $E_\omega$-vector space $E_\omega\otimes_{\Lambda(\Gamma)} M=0$ as it is annihilated by $p_\omega(\xi)\not=0$.
Similarly, as $p_\omega(f_i)\not=0$ for $i\geq 2$, $E_{\omega}\otimes_{\Lambda(\Gamma)}\Lambda(\Gamma)/(f_i)=0$ .
On the other hand, as  $p_\omega(f_1)=0$, $E_{\omega}\otimes_{\Lambda(\Gamma)}\Lambda(\Gamma)/(f_1)=E_\omega$.
Also, $\ker\be\left[ \le \iota_\omega\le \right]=0$ as it is annihilated by $p_\omega(\xi)$.
Therefore, $\iota_\omega$ is an isomorphism between $E_\omega^{c_1}$ and $E_\omega\otimes_{\Lambda(\Gamma)} X_L^0$. Hence
$$\rank_{\O_\omega} \O_\omega\otimes_{\Lambda(\Gamma)} X_L^0=\dim_{E_\omega} E_\omega\otimes_{\Lambda(\Gamma)} X_L^0=c_1.$$
Then we deduce $s(\omega^{-1})=c_1$ by using \eqref{e:news} together with the fact that
$$(\O_\omega\le\bigcup_{F}\e(\Sel_{p^{\infty}}(A/L)^{\Gal(L/F)}\e)_{div}\le)^{(\omega^{-1})}=(\O_\omega X_L^0/\ker  \be \left[\le p_\omega \le\right] X_L^0)^\vee
\simeq (\O_\omega\otimes_{\Lambda(\Gamma)} X_L^0)^\vee.$$
Similarly, by tensoring the exact sequence \eqref{e:xlsim} with $E_\omega$, we get $a_1=s(\omega^{-1})$, whence $a_1=c_1$.
\end{proof}
By Theorem \ref{t:x0} there are non-negative integers
$a_1',...,a_l'$, $a_1'',...,a_l''$ with $a_i''\leq a_i'\leq a_i$, so that $Y_L\sim (\Lambda(\Gamma)/(f_1))^{a_1'}\oplus\cdots \oplus (\Lambda(\Gamma)/(f_l))^{a_l'}$ and $Z_L\sim (\Lambda(\Gamma)/(f_1))^{a_1''}\oplus\cdots \oplus (\Lambda(\Gamma)/(f_l))^{a_l''}$.

\end{subsection}

\begin{subsection}{Algebraic functional equations}\label{su:algfun}
Let ${}^{\sharp}:\Lambda(\Gamma)\longrightarrow \Lambda(\Gamma)$, $x\mapsto x^{\sharp}$, denote the $\Z_p$-algebra
isomorphism induced by the involution $\gamma\mapsto \gamma^{-1}$, $\gamma\in \Gamma$.
For each $\Lambda(\Gamma)$-module $W$, let $W^{\sharp}$ denote the $\Lambda(\Gamma)$-module with the same underlying abelian group
as $W$, while $\Lambda(\Gamma)$ acting via the isomorphism ${}^{\sharp}$. For a simple element $f$ we have $(\Lambda(\Gamma)/(f))^\sharp=\Lambda(\Gamma)/(f)$. Thus, if $X_L$ is torsion, the we have the
functional equations ${X_L^0}^{\sharp}\sim X_L^0$, ${Y_L}^{\sharp}\sim Y_L$ and $Z_L^\sharp\sim Z_L$ as well.
Taking the projective limit over $F$ of the dual of
$$\xymatrix{0\ar[r] & \Sel_{p^{\infty}}(A/F)_{div}\ar[r] & \Sel_{p^{\infty}}(A/F)\ar[r]& \Sha(A/F)_p/\Sha(A/F)_{div}\ar[r] & 0,}$$
where $\Sha(A/F)$ denote the Shafarevich-Tate group, we obtain the exact sequence
$$\xymatrix{0\ar[r] & \mathfrak{a}_L\ar[r] &  X_L\ar[r] & Y_L\ar[r] & 0,}  $$
where
$$\mathfrak{a}_L:=\varprojlim_{F} (\Sha(A/F)_p/\Sha(A/F)_{div})^{\vee}.$$
Then by using the Cassels-Tate pairing on each $\Sha(A/F)_p\times \Sha(A^t/F)_p$, one can actually prove the pseudo-isomorphisms
$\mathfrak{a}_L^{\sharp}\sim \mathfrak{a}_L$ and
$X_L^{\sharp}\sim X_L$. The proof is given in \cite{lltt13}, in which the content of Theorem \ref{t:x0} actually plays a key role.

\subsubsection{$\mathfrak{a}_L$ is torsion}
The following is proved in \cite{lltt13} by using Theorem \ref{t:root}, while if every $v\in S$ is a good ordinary place,
then it can be proved by the control theorem.
\begin{theorem}\label{t:a1}
The module $\mathfrak{a}_L$ is finitely generated and torsion over $\Lambda(\Gamma)$.
\end{theorem}
\end{subsection}
\begin{subsection}{When is $X_M$ torsion?}\label{su:when} For convenience, call an intermediate extension $M$ of
$L/K$ simple, if $\Gal(M/K)\simeq \Z_p^c$, for some $c$.
For such $M$, by repeatedly applying Theorem \ref{t:compatible}, we deduce
\begin{equation}\label{e:d-c}
p_{L/M}(\Theta_L)\cdot \varrho
=\Theta_{M}\cdot \vartheta,
\end{equation}
where $\varrho\not=0$ and $\vartheta$ is a product of local factors obtained from those $\vartheta_v$'s in Theorem \ref{t:compatible}.
It is easy to see that $\vartheta\not=0$ unless $M$ is fixed by the decomposition subgroup $\Gamma_v$ of some split-multiplicative place $v\in S$.
Thus, the following theorem is proved by taking
$$\mathfrak{T}=\{L^{\Gamma_v}\;\mid \; v\in S \; \text{is a split-multiplicative place}\}.$$
\begin{theorem}\label{t:nontor} Suppose $X_L$ is non-torsion. If $L/K$ only ramifies at good ordinary places, then $X_M$ is non-torsion,
for every simple intermediate extension $M$. In general, there is a finite set $\mathfrak{T}$ consisting of proper simple intermediate extensions of $L/K$, such that $X_M$ is non-torsion unless $M\subset M_j$ for some $M_j\in \mathfrak{T}$.
\end{theorem}

By \eqref{e:d-c}, if $\Theta_{M}=0$, then $p_{L/M}(\Theta_L)=0$. 
Put $T_M:=\{\omega\in \widehat{\Gamma}\;\mid\; \omega(\gamma)=1, \;\text{for all}\; \gamma\in\Gal(L/M)\}$,
the $\Z_p$-flat of codimension $d-c$ determined by $M$. Then $p_{L/M}(\Theta_L)=0$ if and only
if $T_M\subset \Delta_{\Theta_L}$, or equivalently, $T_M\subset T_j$ for some $j$ if $\Delta_{\Theta_L}=\bigcup_{j=1}^\nu T_j$. Let $M_j$ be the maximal
simple intermediate extension of $L/K$, so that $T_{M_j}\subset T_j$. Then the following theorem is proved by setting
$$\mathfrak{T}=\{M_j\;\mid\; j=1,...,\nu\}.$$
\begin{theorem}
Suppose $X_L$ is torsion. Then there is a finite set $\mathfrak{T}$ consisting of proper simple intermediate extensions of $L/K$,
such that for each simple intermediate extension $M$ outside $ \mathfrak{T}$, $X_M$ is torsion.
\end{theorem}
Hence, if $d=2$ and $X_L$ is non-torsion (resp. torsion), then $X_M$ is non-torsion (resp. torsion)
for almost all intermediate $\Z_p$-extensions $M$.

\end{subsection}

\begin{subsection}{The growth of $s_n$}\label{su:mw} Let $K_n$ denote the $n$th layer of $L/K$ and
write $I_n$ for the kernel of $\xymatrix{\Lambda(\Gamma) \ar@{->>}[r] & \Z_p[\Gamma(K_n)]}$.
Then $\Sel_{p^{\infty}}(A/L)^{\Gal(L/K_n)}$ is the Pontryagin dual of $X_L/I_n X_L$, whence cofinitely generated over $\Z_p$.
Let $s_n$ denote its corank. Theorem \ref{t:mw} below gives an asymptotic formula of $s_n$. The following lemma as well as its proof
is by I. Longhi. Denote $E_n=\Q_p(\mu_{p^n})$, $\mathcal{E}_n=p^{n(d-2)}$, for $d\geq 2$; $\mathcal{E}_n=1$, for $d=1$.
Let $J$ be an ideal of $\Lambda(\Gamma)$.
\begin{lemma}{\em{(Longhi)}}\label{l:longhi}
If $J=(f_{\gamma,\zeta}^m)$ for some positive integer $m$ and $\delta_\zeta:=[\Q_p(\zeta):\Q_p]$ then
\begin{equation}\label{e:assytoticf}
\rank_{\Z_p}\Lambda/(I_n+J)=\delta_\zeta\cdot p^{n(d-1)} \text{ for } n\gg0\,.
\end{equation}
If $J=(f)$, $f$ not divided by any simple element, or $(f,g)\subset J$, for some relatively prime $f$, $g$, then
\begin{equation}\label{e:assytoticfg}
\rank_{\Z_p}\Lambda/(I_n+J)=O(\mathcal{E}_n).
\end{equation}
\end{lemma}
\begin{proof}
Write $\Gamma_n=\Gamma(K_n)$ and $V_n:=E_n\otimes_{\Z_p}(\Lambda/I_n)$. Then
$$\rank_{\Z_p}\Lambda/(I_n+J)=\dim_{E_n}E_n\otimes_{\Z_p}\big(\Lambda/(I_n+J)\big)=\dim_{E_n}(V_n/JV_n)\,.$$
One has a decomposition of $E_n$-vector spaces
$$V_n=\oplus_{\omega\in \widehat{\Gamma}_n} V_n^{(\omega)}.$$
Moreover $\dim_{E_n}V_n^{(\omega)}=1$ because $V_n\simeq E_n[\Gamma_n]$ is the regular representation.
Obviously
$$JV_n^{(\omega)}=\begin{cases}
0 & \text{if } p_\omega(J)=0 \\
V_n^{(\omega)} & \text{if } p_\omega(J)\neq0\,. \end{cases}$$
Denote $\Delta_J=\{\omega\in \widehat{\Gamma} \;\mid\; p_\omega(J)=0\}$. Then
$$\dim_{E_n}(V_n/JV_n)=|\{\omega\in \widehat{\Gamma}_n:p_\omega(J)=0\}|=|\Delta_J\cap\widehat{\Gamma}_n|=|\Delta_J[p^n]|$$
(the last equality comes from $\widehat{\Gamma}_n=\widehat{\Gamma}[p^n]$).
Here $G[p^n]$ denotes the $p^n$ torsion subgroup of $G$.
Monsky's theorem yields $\Delta_J=\bigcup T_j$, where the $T_j$'s are $\Z_p$-flats. Besides, by \cite[Lemma 1.6]{mon81},
$$
|T_j[p^n]|=p^{n(d-k_j)} \text{ for } n\gg0,
$$
where $k_j$ denotes the codimension of $T_j$.
If $J=(f_{\gamma,\zeta}^m)$, then every $T_i$ is of codimension 1. Hence
\begin{equation*}
\dim_{E_n}(V_n/(f_{\gamma,\zeta}^m)V_n)=|\Delta_{(f_{\gamma,\zeta}^m)}[p^n]|=\delta_\zeta\cdot p^{n(d-1)} \text{ for } n\gg0\,.
\end{equation*}
To show the second assertion we observe that every $T_i$ should be of codimension
greater than $1$, because if some $T_j=T_{\gamma;\zeta}$, then $\Delta_{f_{\gamma,\zeta}}\subset \Delta_J$, whence $f_{\gamma,\zeta}$ divides all elements of $J$. Thus,
\begin{equation*}
\dim_{E_n}(V_n/JV_n)=|\Delta_J[p^n]|=O(\mathcal{E}_n).
\end{equation*}
\end{proof}

\begin{theorem}\label{t:mw} There exists a non-negative integer $\kappa_1$ such that
\begin{equation}
s_n= \kappa_1 p^{nd} +O(p^{n(d-1)}).
\end{equation}
$X_L$ is torsion if and only if $\kappa_1=0$, in this case there exists a non-negative integer $\kappa_2$ such that
\begin{equation}\label{e:kappa2}
s_n= \kappa_2 p^{n(d-1)}+O(\mathcal{E}_n).
\end{equation}
$X_L^0$ is pseudo-null if and only if $s_n=O(\mathcal{E}_n)$.
\end{theorem}
\begin{proof} Suppose $W$ is a finitely generated $\Lambda(\Gamma)$-module. Then \eqref{e:Msim} gives rise to the exact sequence
$$\xymatrix{0\ar[r] & \bigoplus_{i=1}^l (\Lambda(\Gamma)/(f_i^{b_i}))^{a_1}\oplus \bigoplus_{j=1}^m\Lambda(\Gamma)/(\xi_i) \ar[r]^\iota
&  W \ar[r] & M \ar[r] & 0}
$$
for some pseudo-null $M$. Then we tensor the exact sequence with $E_n\otimes_{\Z_p}\Lambda(\Gamma)/I_n$. Since $M$ is annihilated by some relatively prime $f$ and $g$, both $M/I_nM$ and $\Tor_{\Lambda(\Gamma)}(\Lambda(\Gamma)/I_n,M)$ are quotients of some direct sums of finite copies of $V_n/(f,g)V_n$.
Hence, formulae \eqref{e:assytoticf} and \eqref{e:assytoticfg} imply
\begin{equation}\label{e:assytoticw}
\rank_{\Z_p} W/I_nW=\dim_{E_n} E_n\otimes_{\Z_p}  W/I_nW=\sum_i \delta_i p^{n(d-1)}+O(\mathcal{E}_n),
\end{equation}
where $\delta_i=\delta_{\zeta_i}$ if $f_i=f_{\gamma_i,\zeta_i}$. If $X_L$ is torsion, then we take $W=X_L$ to prove \eqref{e:kappa2}. In this case, $X_L^0$ is pseudo-null if and only if  $l=0$ (see Theorem \ref{t:x0}) which means $\sum_i \delta_i=0$.

Suppose $X_L$ is non-torsion. Then $\Delta_{\Theta_L}=\widehat{\Gamma}$, whence by Theorem \ref{t:root}, $s(\omega)>0$ for all $\omega$.
Since \eqref{e:amorphism} is of finite kernel and cokernel, we have
$$\rank_{\Z_p} X_L^0/I_nX_L^0=\corank_{\Z_p} (\Sel_{p^{\infty}}(A/L)^{\Gal(L/K_n)})_{div}\geq p^{nd}.$$
By \eqref{e:assytoticw}, $X_L^0$ is non-torsion, and hence not pseudo-null.
Let $x_1,...,x_{\kappa_1}\in X_L$ form a basis of the vector space generated by $X_L$ over the field of fractions of $\Lambda(\Gamma)$. Then we have an exact sequence
$$\xymatrix{0\ar[r] & \sum_i \Lambda(\Gamma)\cdot x_i \ar[r] & X_L \ar[r] & W \ar[r] & 0,}$$
where $W$ is a torsion $\Lambda(\Gamma)$-module. Then we tensor the exact sequence with $E_n\otimes_{\Z_p}\Lambda(\Gamma)/I_n$.
\end{proof}
\begin{remark}{\em{
By a similar argument, one can prove that:}} {\em{(1)}} There exists a finite number of $Z_p$-flats
$T_1,...,T_l$ so that $s(\omega)=\kappa_1$ for each $\omega\not\in \bigcup_i T_i$. {\em{(2)}} If $X_L$ is torsion, then for each $\Z_p$-flat $T\subset \Delta_{\Theta_L}$, there is a finite number of proper $Z_p$-flats
$T_1',...,T_\nu'\subset T$ so that $s(\omega)$ is a constant for each $\omega\in T-\bigcup_i T_i'$. {\em{(3)}} There is a bound of $s(\omega)$ for all $\omega\in \widehat{\Gamma}$.

{\em{Also, Theorem \ref{t:mw} generalizes \cite[Proposition 1.1]{maru03}, as we have $s_n=\rank A(K_n)$ if $L/K$ only ramifies at good ordinary places and $\Sha(A/K_n)_p$ is finite.  }}
\end{remark}
\end{subsection}
\end{section}


\begin{section}{preliminary}\label{s:pre}
In this section, we assume that $\Gamma\simeq \Z_p^d$, with $d\geq 0$,
except in Lemma \ref{l:coco} and
Lemma \ref{l:cofito}, where we assume $d\geq 1$. If $d=0$, we set $\Lambda(\Gamma)=\Z_p$; otherwise,
$\Gal(L/L')=\Psi\simeq\Z_p$ is topologically generated by $\psi$.
Let $F$ denote a finite intermediate extension of $L/K$.

\begin{subsection}{The characteristic ideal}\label{su:char}
Let $W$ be a finitely generated $\Lambda(\Gamma)$-module.
Recall that $W$ is pseudo-null if and only if
there are relatively prime elements
$f_1,...,f_n$, $n\geq 2$, in $\Lambda(\Gamma)$ so that
$f_i\cdot W=0$ for every $i$.
If $W$ is non-torsion over $\Lambda(\Gamma)$, then
$$\chi_{\Lambda(\Gamma)}(W)=0.$$
If $W$ is torsion, then there exist irreducible $\xi_1,...,\xi_m\in\Lambda(\Gamma)$,
and a pseudo-isomorphism
$$\xymatrix{\phi:\bigoplus_{i=1}^m
\Lambda(\Gamma)/\xi_i^{r_i}\Lambda(\Gamma)\ar[r] & W,}$$
\cite[\S4.4, Theorem 5]{bou72}. In this case, the
characteristic ideal is
$$\chi_{\Lambda(\Gamma)}(W)=\prod_{k=1}^m (\xi_i)^{r_i}\not=0.$$
Denote
$$[W]=\bigoplus_{i=1}^m \Lambda(\Gamma)/\xi_i^{r_i}\Lambda(\Gamma).$$
Since each non-zero element in $[W]$
cannot be simultaneously annihilated by relatively primes elements in
$\Lambda(\Gamma)$, there is no non-trivial pseudo-null
submodule of $[W]$, and hence $\phi$ is an embedding.
We shall fix a exact sequence (with $N$ pseudo-null):
\begin{equation}\label{e:pseu}
\xymatrix{0\ar[r] &  [W]\ar[r] & W
\ar[r] &  N\ar[r] & 0.}
\end{equation}

\begin{lemma}\label{l:coco} Suppose $M$ is a $\Lambda(\Gamma)$-module and
$\Gamma_0$ is a closed subgroup of $\Gamma$ such that the composition $\xymatrix{\Gamma_0\ar@{^{(}->}[r] &
\Gamma \ar@{->>}[r] & \Gamma'}$ is an isomorphism.
If $M$ considered as a $\Lambda(\Gamma_0)$-module is finitely generated and torsion, then
$$\chi_{\Lambda(\Gamma')}(M^{\Psi})
=\chi_{\Lambda(\Gamma')}(M/(\psi-1)M).$$
\end{lemma}
\begin{proof} It follows from the exact sequence of $\Lambda(\Gamma_0)$-modules:
$$\xymatrix{0 \ar[r] &  M^{\Psi}\ar[r] & M \ar[r]^{\psi-1} &  M\ar[r] &  M/(\psi-1)M\ar[r] & 0.}$$
\end{proof}
\end{subsection}

\begin{subsection}{The Hochschild-Serre spectral sequence}\label{su:hs}
In this subsection, let $\mathcal{K}$ be any field.
Let $\mathfrak{S}$ be a sheaf of abelian groups on the flat topology of $\mathcal{K}$ and let
$\mathcal{F}/\mathcal{K}$ be a finite Galois extension with $G=\Gal(\mathcal{F}/\mathcal{K})$.
Then there is the Hochschild-Serre
spectral sequence
\begin{equation}\label{e:spect}
E_2^{p,q}=\coh^p(G,\coh^q(\mathcal{F},\mathfrak{S}))
\Longrightarrow \coh^{p+q}(\mathcal{K},\mathfrak{S}),
\end{equation}
and in particular (see \cite[p.105]{mil80}), the exact sequence
\begin{equation}\label{e:infres}
\xymatrix{0\ar[r] & \coh^1(G,\mathfrak{S}(\mathcal{F}))\ar[r]^{inf^1_{F/\mathcal{K}}} & \coh^1(\mathcal{K},\mathfrak{S})\ar[r]^{res^1_{\mathcal{F}/\mathcal{K}}} &
\coh^1(\mathcal{F},\mathfrak{S})^G\ar[lld]_{d_{\mathcal{F}}^{0,1}} &  \\
 & \coh^2(G,\mathfrak{S}(\mathcal{F}))\ar[r]_{inf^2_{\mathcal{F}/\mathcal{K}}} &
 \ker \be\left[\e res^2_{\mathcal{F}/\mathcal{K}}\right]\ar[r] & \coh^1(G,\coh^1(\mathcal{F},\mathfrak{S}))\ar[r]_-{d_{\mathcal{F}}^{1,1}}
 & \coh^3(G,\mathfrak{S}(\mathcal{F})),}
\end{equation}
where for each $i$, $inf^i_{\mathcal{F}/\mathcal{K}}:\coh^i(G,\mathfrak{S}(\mathcal{F}))\longrightarrow \coh^i(\mathcal{K},\mathfrak{S})$ and $res^i_{\mathcal{F}/\mathcal{K}}:\coh^i(\mathcal{K},\mathfrak{S})\longrightarrow \coh^i(\mathcal{F},\mathfrak{S})$ denote
the inflation and the restriction maps.

\begin{lemma}\label{l:spc}
Suppose $\mathcal{K}\subset \mathcal{E}\subset \mathcal{F}$ and $\mathcal{E}/\mathcal{K}$ is Galois
with $\Gal(\mathcal{F}/\mathcal{E})=H$.
Then we have commutative diagrams:
\begin{equation}\label{e:spcoh1}
\xymatrix{
\coh^1(\mathcal{K},\mathfrak{S})\ar@{=}[d] \ar[r]^{res^1_{\mathcal{E}/\mathcal{K}}}
& \coh^1(\mathcal{E},\mathfrak{S})^{G/H} \ar[r]^{d_{\mathcal{E}}^{0,1}}\ar[d]^{res_{\mathcal{F}/\mathcal{E}}^{0,1}}
& \coh^2(G/H,\mathfrak{S}(\mathcal{E}))\ar[d]^{inf_{\mathcal{F}/\mathcal{E}}^2}\\
\coh^1(\mathcal{K},\mathfrak{S}) \ar[r]^{res^1_{\mathcal{F}/\mathcal{K}}}
& \coh^1(\mathcal{F},\mathfrak{S})^G \ar[r]^-{d_{\mathcal{F}}^{0,1}}
& \coh^2(G,\mathfrak{S}(\mathcal{F})),}
\end{equation}
and
\begin{equation}\label{e:spcoh3}
\xymatrix{
\ker \be\left[\e res_{\mathcal{E}/\mathcal{K}}^2\right]\ar[r]\ar[d] & \coh^1(G/H,\coh^1(\mathcal{E},\mathfrak{S})) \ar[r]^-{d_{\mathcal{E}}^{1,1}}\ar[d]^{res_{\mathcal{F}/\mathcal{E}}^{1,1}} & \coh^3(G/H,\mathfrak{S}(\mathcal{E}))\ar[d]^{inf_{\mathcal{F}/\mathcal{E}}^3}\\
\ker \be\left[\e res_{\mathcal{F}/\mathcal{K}}^2\right]\ar[r] & \coh^1(G,\coh^1(\mathcal{F},\mathfrak{S}))\ar[r]^-{d_{\mathcal{F}}^{1,1}}\ar[r] &
\coh^3(G,\mathfrak{S}(\mathcal{F})),}
\end{equation}
where
$$res_{\mathcal{F}/\mathcal{E}}^{i,j}:\coh^i(G/H,\coh^j(\mathcal{E},\mathfrak{S}))
\longrightarrow \coh^i(G,\coh^j(\mathcal{F},\mathfrak{S}))$$
is induced by the restriction map
$\coh^j(\mathcal{E},\mathfrak{S})\longrightarrow \coh^j(\mathcal{F},\mathfrak{S})$ that
respects the actions of $G/H$ {\em{(}}on the left-hand side {\em{)}} and $G$
 {\em{(}}on the right-hand side {\em{)}}.
\end{lemma}
\begin{proof}
Recall that if the complex $C$ is an injective resolution
of $\mathfrak{S}$ (in the category
of sheaves on the flat site of $\text{Spec} \,\mathcal{K}$)
and the bi-complex $I$ is a fully injective (Cartan-Eilenberg) resolution
of $C(\mathcal{F})$ (in the category of $G$-modules), then the spectral sequence (\ref{e:spect}) is obtained from the
bi-complex $I^G$ (with $(I^G)^{pq}=(I^{pq})^G$, the part of $I^{pq}$
fixed by $G$).
Let the bi-complex $J$ be a fully injective resolution
of $C(\mathcal{E})$ (in the category of $G/H$-modules) that gives rise to
the spectral sequence
$$
E_2^{p,q}=\coh^p(G/H,\coh^q(\mathcal{E},\mathfrak{S}))
\Longrightarrow \coh^{p+q}(\mathcal{K},\mathfrak{S}).
$$
Since an injective
$G/H$-module is also injective as $G$-module, we have a $G$-morphism
$J\longrightarrow I$. The commutative diagrams are obtained from the induced
morphism $J^{G/H}\longrightarrow I^G$.
\end{proof}
We are mostly interested in the case where $G$ is cyclic.
Denote ${g}=|G|$, $h=|H|$. We fix a generator of $G$ and choose its $H$-coset
as a generator of $G/H$. Then we have the commutative diagram (see \cite[VIII.4]{ser79}):
$$
\xymatrix{
\coh^2(G/H,\Z) \ar@{=}[r]\ar[d]^{inf^2} & \Hom(G/H,\Q/\Z)  \ar@{=}[r] & \Z/\frac{g}{h}\Z\ar[d]\\
\coh^2(G,\Z) \ar@{=}[r] & \Hom(G,\Q/\Z) \ar@{=}[r] & \Z/g\Z,}
$$
where the right down-arrow is induced by the homomorphism
$\Z\longrightarrow \Z$, $1\mapsto h$.
Let $\delta_G\in \coh^2(G,\Z)$ be the class corresponding to $1\pmod{g\Z}$
in the above diagram.
Then we have the induced commutative
diagram:
\begin{equation}\label{e:inf2}
\xymatrix{
\mathfrak{S}(\mathcal{K})/\Nm_{G/H}(\mathfrak{S}(\mathcal{E})) \ar@{=}[r]\ar[d]
& {\widehat{\coh}}^0(G/H,\mathfrak{S}(\mathcal{E})) \ar[r] & \coh^2(G/H,\mathfrak{S}(\mathcal{E}))\ar[d]^{inf^2_{\mathcal{F}/\mathcal{E}}}\\
\mathfrak{S}(\mathcal{K})/\Nm_{G}(\mathfrak{S}(\mathcal{F})) \ar@{=}[r] & {\widehat{\coh}}^0(G,\mathfrak{S}(\mathcal{F})) \ar[r]
& \coh^2(G,\mathfrak{S}(\mathcal{F})),}
\end{equation}
where the upper and lower right-arrows are, respectively,
cup-product with $\delta_{G/H}$ and $\delta_G$, and the left
down-arrow is induced by the multiplication by $h$ on
$\mathfrak{S}(\mathcal{K})$. Similarly, we have the commutative diagram
\begin{equation}\label{e:inf3}
\xymatrix{
\coh^1(G/H,\mathfrak{S}(\mathcal{E})) \ar[r]\ar[d] &
\coh^3(G/H,\mathfrak{S}(\mathcal{E}))\ar[d]^{inf^3_{\mathcal{F}/\mathcal{E}}}\\
\coh^1(G,\mathfrak{S}(\mathcal{F})) \ar[r] &
\coh^3(G,\mathfrak{S}(\mathcal{F})),}
\end{equation}
where the left down-arrow is the $h$ multiple of $inf^1_{\mathcal{F}/\mathcal{E}}$.
\end{subsection}

\begin{subsection}{Cohomology groups of $A[p^{\infty}]$}\label{su:torsion} In this section
let $D$ be a discrete $p$-primary abelian group cofinitely generated over $\Z_p$.

\subsubsection{} Assume that $\Gamma$ acts continuously on $D$.
\begin{lemma}\label{l:cofito}
Suppose $D^{\Gamma}$ is finite.
Write
$C_1=\coh^1(\Psi,D)$, $C_2=D^{\Psi}$.
Then $C_1^{\vee}$ is finitely generated over $\Z_p$ and
$$\chi_{\Lambda(\Gamma')}({C_1}^\vee)
=\begin{cases}
\chi_{\Lambda(\Gamma')}({C_2}^\vee), & \;\text{if} \; d\geq 2;\\
\chi_{\Lambda(\Gamma')}(C_2/C_2\cap D_{div}), & \;\text{if} \; d=1.\\
\end{cases}
$$
\end{lemma}
\begin{proof}
Denote $\Psi^{(n)}=\Psi^{p^n}$ and $\Psi_n=\Psi/\Psi^{(n)}$.
Every $x\in D$ is fixed by $\Psi^{(n)}$, for some $n$.
If $x$ is of order $p^m$, then, for $l\geq n+m$,
the norm $\Nm_{\Psi_{l}}(x)=0$, and hence $x$ determines a class
in $\coh^1(\Psi_{l},D^{\Psi^{(l)}})$.
Therefore, $C_1=D/(\psi-1)D$ and consequently, ${C}_1^\vee=
({D}^\vee)^{\Psi}$. On the other hand, ${C}_2^\vee
=D^\vee/(\psi-1){D}^\vee$. If $d\geq 2$, then the lemma follows from
Lemma \ref{l:coco}, as ${D}^\vee$, being finitely generated over $\Z_p$, is torsion over $\Lambda(\Gamma_0)$
for any choice of $\Gamma_0$.
If $d=1$, then $\Psi=\Gamma$ and $D^{\Psi}$ is finite.
It follows that
$D_{div}\longrightarrow D_{div}$, $x\mapsto (\psi-1)x$, is surjective.
Thus, the snake lemma for the multiplication of $(\psi-1)$
on the exact sequence
$$0\longrightarrow D_{div}\longrightarrow D\longrightarrow {\bar D}
\longrightarrow 0$$
gives rise to the exact sequence
$$0\longrightarrow D_{div}^{\Psi}\longrightarrow C_2\longrightarrow
{\bar D}^{\Psi}\longrightarrow 0$$
as well as the isomorphism
$$C_1\simeq {\bar D}/(\psi-1){\bar D}.$$
Since ${\bar D}$ is torsion over $\Lambda(\Gamma_0)=\Z_p$, for $\Gamma_0=0=\Gamma'$,
Lemma \ref{l:coco} and the above
exact sequences imply $\chi_{\Z_p}(C_1)=\chi_{\Z_p}({\bar D}^{\Psi})=\chi_{\Z_p}(C_2/C_2\cap D_{div})$.
Now, ${C}_1^\vee\simeq C_1$, as $C_1$ is finite.
\end{proof}
\begin{corollary}\label{c:chicoh1} If $d\geq 3$, then ${\coh^1(\Psi,A[p^{\infty}](L))}^\vee$ is pseudo-null over $\Lambda(\Gamma')$. In general,
$$\chi_{\Lambda(\Gamma')}({\coh^1(\Psi,A[p^{\infty}](L))}^\vee)
=\begin{cases}
\chi_{\Lambda(\Gamma')}({A[p^{\infty}](L')}^\vee), & \;\text{if} \; d\geq 2;\\
(\frac{|A[p^{\infty}](K)|}{|A[p^{\infty}](K)\cap A[p^{\infty}](L)_{div}|}), & \;\text{if} \; d=1.\\
\end{cases}
$$
\end{corollary}
\begin{proof} The second assertion is immediately from Lemma \ref{l:cofito}.
Since ${A[p^{\infty}](L')}^\vee$ is finitely generated over $\Z_p$, it is
pseudo-null over $\Lambda(\Gamma')$, if $d\geq 3$, whence the first assertion follows.
\end{proof}
\begin{corollary}\label{c:cofito}
Suppose $\Gamma\simeq \Z_p$ and $D$ is finite.
Then
$$|\coh^1(\Gamma, D)|=|D^{\Gamma}|.$$
\end{corollary}

\subsubsection{} Next, consider the case where a topological group
$\mathcal{C}\simeq\Z_p$ acts continuously on $D$. We are going to discuss some associated modules over
$\Lambda(\mathcal{C}):=\Z_p[[\mathcal{C}]]$ as well as their characteristic ideals.

Denote $\mathcal{C}^{(m)}=\mathcal{C}^{p^m}$ and $D_m=D^{\mathcal{C}^{(m)}}$ for $m=0,1,...,\infty$. Also, denote $\mathcal{G}^{m'}_m=\mathcal{C}^{(m)}/\mathcal{C}^{(m')}$ for $m'=m,....,\infty$. Define
$\mathrm{M}:=\varprojlim_m D_m$, with the limit taken over norm maps $\Nm_{\mathcal{G}^{m'}_m}:D_{m'}\longrightarrow D_m$.
Define $\mathrm{N}:=D^\vee$, the Pontryagin dual of $D$, and $\mathrm{T}:=\varprojlim_n D[p^n]$, the Tate-module of $D$.

Let $c$ be a topological generator of $\mathcal{C}$. Let $\epsilon_1,...,\epsilon_m$, counted with multiplicities, be the eigenvalues of the action of $c$ on $\mathrm{T}$.
Note that $\epsilon_1\cdots\epsilon_m=\det (c)$ is a $p$-adic unit, and hence so is each $\epsilon_i$.
Also, $\epsilon_j=1$ for some $j$ if and only if $D_0=D^\mathcal{C}$ is infinite.

\begin{lemma}\label{p:torcomp}
We have $\chi_{\Z_p}((D_0\cap D_{div})^\vee)=\prod_{j=1}^m(1-\epsilon_j)=\prod_{j=1}^m(1-\epsilon_j^{-1}).$
\end{lemma}
\begin{proof}
Let $\mathrm{N}^0$ denote the $\Z_p$ free part of $\mathrm{N}$. Then
$(D_0\cap D_{div})^\vee=\mathrm{N}^0/(c-1)\mathrm{N}^0$ whose $p$-adic valuation is the same as that
of the determinant (of $1-c$ acting on $\Q_p\otimes_{\Z_p} \mathrm{N}^0$)
$$\det (1-c)=\prod_{j=1}^m(1-\epsilon_j^{-1})=\pm\prod_{j=1}^m\epsilon_j^{-1}(1-\epsilon_j).$$

\end{proof}

\begin{proposition}\label{p:torcomp1} Suppose $D_m$ is finite for each $m$. 
Then $M\sim T$ as $\Lambda(\mathcal{C})$-modules, and
$$\chi_{\Lambda(\mathcal{C})}(\mathrm{M})=\chi_{\Lambda(\mathcal{C})}(\mathrm{T})=(\prod_{j=1}^m(1-\epsilon_j^{-1}c)),$$
$$\chi_{\Lambda(\mathcal{C})}(\mathrm{N})=(\prod_{j=1}^m(1-\epsilon_j^{-1}c^{-1})).$$
\end{proposition}
\begin{proof}
Let $V=\Q_p\otimes_{\Z_p} \mathrm{T}$
and let $f(s)=\prod_{j=1}^m(s-\epsilon_j)\in\Z_p[s]$ be the characteristic polynomial of the matrix given by the action
of $c$ on $V$. By the Jordan decomposition of the matrix, we can find a
$\Lambda(\mathcal{C})$-submodule $\mathrm{T}'$ of $\mathrm{T}$ of finite index so that
$\chi_{\Lambda(\mathcal{C})}(\mathrm{T}')=(f(c))$. As $\mathrm{T}'$ and $\mathrm{T}$
are pseudo-isomorphic, $\chi_{\Lambda(\mathcal{C})}(\mathrm{T})=(f(c))=(\prod_{j=1}^m(c-\epsilon_j))$ as well. This yields the desired expression of $\chi_{\Lambda(\mathcal{C})}(\mathrm{T})$, since $c$ and $\epsilon_1\cdot \cdots\cdot\epsilon_m$ are units of
$\Lambda(\mathcal{C})$.
Similarly, since $\epsilon_1^{-1},...,\epsilon_m^{-1}$ are the eigenvalues of the action of $c$ on $\Q_p\otimes_{\Z_p} \mathrm{N}$,
we have $\chi_{\Lambda(\mathcal{C})}(\mathrm{N})=(\prod_{j=1}^m(c^{-1}-\epsilon_j))$.

To complete the proof, we need to establish a pseudo-isomorphism $\iota:M\longrightarrow T$.
Without loss of generality we may assume that $D$ is $p$-divisible, as replacing $D$ by $D_{div}$ does not alter the structures of
$M$ and $T$. The exact sequence (of $\cG_m^{m'}$-modules)
$$\xymatrix{0\ar[r] & D_{m'}[p^n]\ar[r] & D_{m'} \ar[r]^{p^n} & p^nD_{m'} \ar[r] & 0},$$
where $D_{m'}[p^n]$ denotes the $p^n$-torsion subgroup of $D_{m'}$, induces the
Kummer exact sequence
$$\xymatrix{ D_m\ar[r]^-{p^n} & D_m\cap p^n D_{m'} \ar[r]^-{\lambda^{m'}_{m,n}} & \coh^1(\cG_m^{m'},D_{m'}[p^n])\ar[r] & \coh^1(\cG_m^{m'},D_{m'})}.$$
Choose $m_0$ so that $D[p^2]\subset D_{m-1}$, for $m\geq m_0$. Then for $m'\geq m\geq m_0$,
$$
\coh^1(\cG^{m'}_{m-1},D_{m'}[p^i])=\Hom(\cG^{m'}_{m-1},D_{m'}[p^i]),\; i=1,2.
$$
Since $D$ is $p$-divisible, $D_{m-1}=D_{m-1}\bigcap p D$. The commutative diagram of exact sequences
$$
\xymatrix{ D_{m-1} \ar[r]^-p \ar@{=}[d] &  D_{m-1}\bigcap p D_m \ar[r]^-{\lambda^m_{m-1,1}}\ar@{^{(}->}[d] & \Hom (\cG^{m}_{m-1},D_m[p]) \ar@{=}[d]\ar[r] & \coh^1(\cG^m_{m-1},D_{m}) \ar@{^{(}->}[d]^-{inf} \\
  D_{m-1} \ar[r]^-{p}  & D_{m-1} \ar[r]^-{\lambda^{\infty}_{m-1,1}} & \Hom (\cG^{\infty}_{m-1},D[p]) \ar[r] & \coh^1(\cG_{m-1}^{\infty},D)},
$$
where {\em{inf}} denotes the inflation map and the second equality is due to the fact that $\cG^{m}_{m-1}$ is the unique quotient group
of $\cG_{m-1}^{\infty}$ of order $p$,
implies $D_{m-1}\subset p D_{m}$. We claim the inclusion $pD_{m}\subset D_{m-1}$ also holds. Therefore, $D_{m-1}=pD_m$.

To prove the claim we show that if $\overline{D}=D_m/D_{m-1}$ then $\overline{D}[p^2]=\overline{D}[p]$, as the assertion implies
$\overline{D}=\overline{D}[p]$. Write $E_i=D_{m-1}\bigcap p^i D_m$ and $F_i= p^iD_{m-1}$. For each $i$ we have the group homomorphism $\pi_i:\overline{D}[p^i]\longrightarrow E_i/F_i$ given by $a\mapsto b$ such that if $a=x$ is the residue class of some $x\in D_m$
modulo $D_{m-1}$, then $b$ is that of $p^i x$ modulo $F_i$. Obviously, every $\pi_i$ is a surjection, and by
$D[p^2]\subset D_{m-1}$, both $\pi_1$ and $\pi_2$ are isomorphisms. Then the desired assertion, and hence the claim, follows
form the fact that the map $E_1/F_1\longrightarrow  E_2/F_2$ induced by $\xymatrix{E_1\ar[r]^-p & E_2}$ is an isomorphism
as shown by the diagram of exact sequences:
$$
\xymatrix{  D_{m-1} \ar[r]^-p \ar@{=}[d] &  D_{m-1}\bigcap p D_m \ar[r]^-{\lambda^m_{m-1,1}}\ar[d]^p & \Hom (\cG^{m}_{m-1},D_m[p])\ar@{=}[d] \ar[r] & \coh^1(\cG^m_{m-1},D_{m}) \ar@{=}[d]\\
 D_{m-1} \ar[r]^-{p^2}  & D_{m-1}\bigcap p^2D_m \ar[r]^-{\lambda^m_{m-1,2}} &  \Hom (\cG^{m}_{m-1},D_m[p^2])\ar[r] & \coh^1(\cG^m_{m-1},D_{m}),}
$$
where the second equality is due to the fact that $\cG^m_{m-1}$ is of order $p$.
Thus, $pD_m=D_{m-1}$ holds for $m\geq m_0$. Therefore, if we fix an $n$ so that $D_{m_0}\subset D[p^{n+m_0}]$, then for all $m \gg 0$
\begin{equation}\label{e:dpmm0}
D[p^{m-m_0+2}]\subset D_{m}\subset D[p^{n+m}].
\end{equation}
Also, for every $a\in D_{m}$, $\tau\in \cG^m_{m-1}$, we have $^{\tau}a-a\in D[p]$. Therefore, the diagram
\begin{equation}\label{e:iotamm-1}
\xymatrix{ D_m \ar@{^{(}->}[r] \ar[d]^{\Nm_{\cG_{m-1}^m}} & D[p^{m+n}] \ar[r]^{p} & D[p^{m+n-1}]\ar[d]^{p}\\
  D_{m-1} \ar@{^{(}->}[r]  & D[p^{m+n-1}] \ar[r]^{p} & D[p^{m+n-2}]}
\end{equation}
is commutative, as we can write for $a\in D_m$
$$\Nm_{\K_m/\K_{m-1}}(a)=pa+\sum_{\tau\in\cG^m_{m-1}}  {}^{\tau}a-a\in pa+D[p].$$

Let $\iota_m$ denote the composition $ D_m\longrightarrow D[p^{m+n}]\longrightarrow D[p^{m+n-1}]$ appearing in the diagram \eqref{e:iotamm-1}.
Then we can define the desired pseudo-isomorphism $\iota:M\longrightarrow  \mathrm{T} $ as the projective limit of $\iota_m$.
The finiteness of  $\ker\le\left[\iota\right]$ and $\coker\le\left[\iota\right]$ follows from the fact that they are the projective limits
of finite abelian groups of bounded order, since
$\ker\le\left[\iota_m\right]\subset D[p]$ and by \eqref{e:dpmm0} $\coker\le\left[\iota_m\right]$ can be viewed as a quotient group of
$D[p^{n+m-1}]/D[p^{m-m_0+1}]\simeq D[p^{n+m_0-2}]$ .
\end{proof}

The above results can be applied to the cases where $\mathcal{C}$ is the Galois group of a field extension $\mathcal{L}/\mathcal{K}$
and $D=B(\mathcal{L})_p$ where $B$ is an abelian variety defined over $\mathcal{K}$. Indeed, if $\mathcal{K}$ is either a global field
or a finite field,
then the condition of proposition \ref{p:torcomp1} is satisfied.
For the rest of this section let $\mathcal{K}$ be a finite field of order $\mathfrak{q}$.
Denote $\tilde{\L}=\L(B[p^{\infty}](\overline{\K}))$ and let
$\tilde{c}$ be a topological generator
$\tilde{\mathcal{C}}:=\Gal(\tilde{\L}/\K)$, sent to $c$ under the natural map $\tilde{\mathcal{C}}\longrightarrow \mathcal{C}$.
Since $\mathrm{T}$ is contained in the Tate module $\tilde{\mathrm{T}}$ of $B[p^{\infty}](\overline{\K})$,
the eigenvalues of the action of $\tilde{c}$ on $\tilde{\mathrm{T}}$ can be expressed as $\epsilon_1,...,\epsilon_m,...,\epsilon_t$.
\begin{proposition}\label{p:torcomp2} Suppose $\K$ is a finite field of order $\mathfrak{q}$ and assume the above notation.
Then
$$\chi_{\Lambda(\mathcal{C})}(\mathrm{M})=\chi_{\Lambda(\mathcal{C})}(\mathrm{T})=(\prod_{j=1}^t(1-\epsilon_j^{-1}c)).$$
and
$$\chi_{\Lambda(\mathcal{C})}(\mathrm{N})=(\prod_{j=1}^t(1-\epsilon_j^{-1}c^{-1})).$$
\end{proposition}
\begin{proof}
Since $\L/\K$ is the maximal pro-$p$ abelian extension of $\K$, $\Gal(\tilde{\L}/\K)=\mathcal{C}\times\mathcal{H}$, where $\mathcal{H}=\Gal(\tilde{\L}/\L)$ is a finite cyclic group of order prime to $p$.
We remark that the action of $\tilde{c}$ on $\tilde{\mathrm{T}}$ is semi-simple.
To see this, we may assume that $B$ is a simple abelian variety and $\tilde{c}=\Frob_{\mathfrak{q}}$,
the Frobenius substitution that sends $x\in \tilde{\L}$ to $x^{\mathfrak{q}}$. Then $\tilde{c}$ and the Frobenius
endomorphism $\mathrm{F}_{\mathfrak{q}}\in \text{End}_{\mathcal{K}}(B)$ give rise to the same action on $\tilde{V}:=\Q_p\otimes_{\Z_p}\tilde{\mathrm{T}}$.
Since $\Z[\tF_{\mathfrak{q}}]$ is an order in a number field, the minimal polynomial $F(s)$ of $\tF_{\mathfrak{q}}$ over $\Q$ is irreducible.
In particular, it has no double root. Then, since $F(\tilde{c})=0$ on $\tilde{V}$, the action of $\tilde{c}$ gives rise to a diagonalizable
matrix.

Choose a positive integer $\nu$ so that $|\mathcal{H}|$ divides $p^{\nu}-1$. Then $\sigma:=\lim_{n\rightarrow\infty}\tilde{c}^{p^{n\nu}}$
is a generator of $\mathcal{H}$. Since $V:=\Q_p\otimes_{\Z_p} \mathrm{T}$ is the $1$-eigenspace of $\sigma$, we see that $\{\epsilon_1,...,\epsilon_m\}$
is exactly the subset consisting of those elements $\epsilon\in \{\epsilon_1,...,\epsilon_t\}$ satisfying $\lim_{n\rightarrow\infty} \epsilon^{p^{n\nu}}=1$ in the $p$-adic topology,
or equivalently, $\ord(\epsilon-1)>0$, where $\ord$ denote the valuation on $\overline{\Q}_p$ with $\ord(p)=1$. This shows that the product $\prod_{j=m+1}^t(1-\epsilon_j)$ is a $p$-adic unit. Hence, $\prod_{j=m+1}^t(1-\epsilon_j^{-1}c)=\prod_{j=m+1}^t(1-\epsilon_j^{-1}+\epsilon_j^{-1}(1-c))$
is a $\Lambda(\mathcal{C})$-unit, and so is $\prod_{j=m+1}^t(1-\epsilon_j^{-1}c^{-1})$.

\end{proof}

\begin{remark}\label{r:tor} If $\K$ is a finite field,
then $B(\L)_p$ is actually $p$-divisible.
Let $\mathcal{H}$ be as in the above proof. Since $|\mathcal{H}|$ is prime to $p$, we have $\coh^1(\mathcal{H}, B[p](\tilde{\L}))=0$
in the Kummer exact sequence
$$\xymatrix{ B[p](\L)\ar@{^{(}->}[r] & B(\L)_p \ar[r]^-p & B(\L)_p \ar[r] & \coh^1(\mathcal{H}, B[p](\tilde{\L}))}.
$$

\end{remark}

\subsubsection{} We end this section by showing that generically $A[p^{\infty}](L)$ is finite.
For related results see \cite{zar87,blv09,vol95}. 

\begin{proposition}\label{p:torl}
In general, $A[p^{\infty}](L)$ is finite, except possibly for the following cases:
\begin{enumerate}
\item[(a)] $K$ is a number field and $A$ contains a nontrivial abelian variety
of CM-type.
\item[(b)] $K$ is a function field, there exists no split multiplicative place of $A$, and $L/K$ contains the constant
$\Z_p$-extension $L_0/K$. 
\end{enumerate}
\end{proposition}
\begin{proof}
(a) is from \cite{zar87}. In the case (b), assume that $A[p^{\infty}](L)$ is infinite.
Then obviously
$L_1:=K(A[p^{\infty}](L))\subset L$ is an infinite pro-$p$ abelian extension of $K$.
We claim that it is everywhere unramified, splitting completely at every split multiplicative place.
Then it follows that $L_0\subset L_1$, as the maximal everywhere unramified pro-$p$ abelian extension of $K$
is a finite extension of $L_0$. Furthermore, it also follows that there is no split multiplicative place
of $A$, as $L_0/K$ does not split completely at any place.

Let $v$ be a split-multiplicative place.
By (\ref{e:desplit}) if $P\in A[p^n](K^s)$, there must be some $Q\in \Omega_v$ so that
$K_v(P)=K_v(Q^{p^{-n}})$. This implies $K_v(P)=K_v$, since it is both separable and purely inseparable over $K_v$.
Thus, $L_1/K$ splits completely at every split-multiplicative place.

Suppose $v$ is a good ordinary place and ${\bar{A}}$ is the reduction of $A$. Then the reduction map induces
the isomorphism of $\Gal({\overline{K}}_v/K_v)$-module (see e.g. \cite[Corollary
2.1.3]{tan10}):
\begin{equation*}
A[p^{\infty}](K^s)\simeq {\bar{A}}[p^{\infty}](\overline{\F}_v).
\end{equation*}
This shows $A[p^{\infty}](L)$ is fixed by the inertia subgroup of $\Gamma_v$, whence
$L_1/K$ is unramified at $v$. Since $L/K$ is only ramified at splits multiplicative places or good ordinary places, the intermediate extension
$L_1/K$ must be everywhere unramified.

\end{proof}

By the definition we see that $\mathrm{w}_{K'/K}=(1)$ if $A[p^{\infty}](K')$ is finite.
\begin{corollary}\label{c:torl}
We have $\mathrm{w}_{K'/K}=(1)$ unless $K$ is a number field and $A$ contains a nontrivial abelian variety
of CM-type, or $K$ is a function field, there exists no split multiplicative place of $A$, and $K'/K$ is the constant $\Z_p$-extension.
\end{corollary}
\end{subsection}

\begin{subsection}{Tate's local duality theorem}\label{sub:tate}
By the Tate's local duality theorem (see \cite{t62}, \cite{mil72}, or \cite[III.7.8]{mil86}),
$\coh^1(K_v,A)_p$, endowed with the discrete topology, is the Pontryagin dual of the $p$-completion
$A^t(K_v)^{\wedge}:=\varprojlim_{n}A^t(K_v)/p^nA^t(K_v)$ endowed with the $p$-adic topology.
If $v$ is archimedean, then $\coh^1(K_v,A)_p$ is trivial unless $K_v=\R$ and $p=2$.
If $v$ is a non-archimedean place $v$, then, since $\coh^1(K_v,A)$ is dual to $A^t(K_v)$
and is the direct product of its $p$-primary part and the non-$p$ part,
$A^t(K_v)^{\wedge}$ can be identified as the largest pro-$p$ closed subgroup of $A^t(K_v)$.
Let
$$<\;,\;>_{K_v}: \coh^1(K_v,A)_p\times A^t(K_v)^{\wedge}
\longrightarrow \Q_p/\Z_p$$
denote the ``$p$-part'' of the local Tate's duality pairing. 
The proof of the following lemma can be found in \cite[Corollary 2.3.3]{tan10}.
\begin{lemma}\label{l:tatenorm} Let $v$ be a place of $K$. Then under the local Tate's duality
the cohomology group $\coh^1(\Gamma_v, A(L_v))\subset \coh^1(K_v, A)_p$ 
equals the annihilator
of $\Nm_{L_v/K_v}(A^t(L_v))$, and hence is
the Pontryagin dual of $A^t(K_v)/\Nm_{L_v/K_v}(A^t(L_v))$.
\end{lemma}

\begin{corollary}\label{c:nf}
If $char.(K)=0$, $\coh^1(\Gamma_v,A(L_v))$ is cofinitely generated over $\Z_p$.
\end{corollary}
\begin{proof}
$A^t(K_v)^{\wedge}$ is finitely generated over $\Z_p$, \cite{mat55}.
\end{proof}

\end{subsection}

\begin{subsection}{The Cassels-Tate exact sequence}\label{sub:ct}
The group $\Gamma(F):=\Gal(F/K)$ acts naturally on
$$\mathcal{H}^i(A/F):=\bigoplus_{\text{all}\;w} \coh^i(F_w,A)_p,$$
for each $i$, and it also acts naturally on $\coh^i(F,A)_p$ so that the localization map
$$loc_F^{\le i}:\coh^i(F,A)_p\longrightarrow
\mathcal{H}^i(A/F)$$
actually respects these actions. The direct product
$$\mathcal{H}^0(A^t/F):=\prod_{\text{all}\; w} A^t(F_w)^{\wedge}$$
is also endowed with a $\Gamma(F)$-action
so that all local parings together define the $\Gamma(F)$-equivariant global perfect pairing
$$
\xymatrix{<\;,\;>_{F}: & \mathcal{H}^1(A/F)\times\mathcal{H}^0(A^t/F) \ar[r] & \Q/\Z\\
& (\eta_v)_v\times (\xi_v)_v \ar@{|->}[r] & \sum_{\text{all}\;v} <\eta_v,\xi_v>_{F_v}}
$$
that identifies $\mathcal{H}^1(A/F)$ with the Pontryagin dual $\mathcal{H}^0(A^t/F)^{\vee}$.
As usual, write $\Sha^i(A/F)=\ker \be\left[\e loc^{\le i}_F\right]$. Then we have the exact sequence
\begin{equation*}
\xymatrix{0\ar[r] & \Sha^1(A/F) \ar[r] & \coh^1(F,A)_p \ar[r]^{loc^{\le 1}_F} & \mathcal{H}^1(A/F) \ar@{=}[r] & \mathcal{H}^0(A/F)^{\vee} .}
\end{equation*}

Recall that for $m=1,...,\infty$, the $p^m$-Selmer group
$\Sel_{p^m}(A/F)$ is the kernel of the composition
\begin{equation}\label{e:loc}
\mathcal{L}_F:\coh^1(F,A[p^m])\longrightarrow
\coh^1(F,A)_p\stackrel{loc_F^1}{\longrightarrow}
\mathcal{H}^1(A/F).
\end{equation}
There exists an injection that identifies
$$\text{T}_p\Sel(A^t/F):=\varprojlim_{m}\Sel_{p^m}(A^t/F)$$
as a subgroup of $\mathcal{H}^0(A^t/F)$
(see Corollary I.6.23(b)
\cite[Proposition 5,6]{mil86}, \cite{mil72}, or \cite{gat07})
and the (generalized) Cassels-Tate exact sequence (\cite{cas64,gat07,t62}) asserts the following.
\begin{mytheorem}\label{t:gct}
The image of $loc^{\le 1}_F$ equals the annihilator of  $\T_p\Sel(A^t/F)$,
whence
$$\coker\be\left[\e loc^{\le 1}_F\right]\simeq \T_p\Sel(A^t/F)^{\vee}.$$
\end{mytheorem}

\begin{proposition}\label{p:coh2st}
The localization map $loc^{\le 2}_F$ is injective and
\begin{equation*}
\mathcal{H}^2(A/F)=\begin{cases}
\bigoplus_{v\;\text{real}} \coh^2(F_v,
A),& \;\text{if}\;char.(F)=0;\\
0, & \;\text{if}\;char.(F)=p.\\
\end{cases}
\end{equation*}
\end{proposition}
\begin{proof}
If $p$ is prime to the characteristic of $K$,
the first statement is proved in \cite[I.6.26(C)]{mil86}); otherwise, a proof is given in \cite{gat12}.
The second statement follows from \cite[I.3.2 and III.7.8]{mil86}.
\end{proof}

\end{subsection}
\end{section}


\begin{section}{The local cohomology groups}\label{s:local} Let $v$ be a place of $K$.
For an intermediate extension $N$ of $L/K$, denote $N'=N\cap L'$, $\Psi(N)=\Gal(N/N')$,
and $\Gamma(N)=\Gal(N/K)$. Let $F$ denote a finite intermediate extension of $L/K$.
Let $\textsf{w}$ (resp. $u$) be a place of $F'$ (resp. $F$) sitting over $v$ (resp. $\textsf{w}$).
We view each $g\in\Gamma(F)$ as an isometry $\xymatrix{F\ar[r]^g_{\sim} & F}$ with the metric of the left-hand side induced by $u$,
while that of the right-hand side induced by $g(u)$ with $|x|_{g(u)}=|g^{-1}(x)|_u$, $\forall x\in F$,
and then extend it to the isomorphism between their completions: $\xymatrix{F_u\ar[r]^-g_-{\sim} & F_{g(u)}}$.
Similarly, the restriction of $g$ gives rise the isomorphism $\xymatrix{F'_{\textsf{w}}\ar[r]^-g_-{\sim} & F'_{g(\textsf{w})}}$,
with $g(u)\mid g(\textsf{w})$, and also, for $i\geq 1$, $g$ induces the isomorphism
\begin{equation}\label{e:isomorphisms}
\xymatrix{\coh^i(\Psi(F)_u, A(F_u))\ar[r]^-g_-{\sim} & \coh^i(\Psi(F)_{g(u)}, A(F_{g(u)}))}.
\end{equation}
By these isomorphisms, we identify $F_u$ and $\coh^i(\Psi(F)_u, A(F_u))$ with $F_{g(u)}$ and $\coh^i(\Psi(F)_{g(u)}, A(F_{g(u)}))$ respectively,
for all $g\in \Psi(F)$, and simply write $F_{\textsf{w}}$ and $\coh^i(\Psi(F)_{\textsf{w}}, A(F_{\textsf{w}}))$ for them.
Then put
$$\mathcal{H}^i_v(A,F/F')=\bigoplus_{\textsf{w}\mid v} \coh^i(\Psi(F)_\textsf{w},A(F_\textsf{w})),$$
endowed with the discrete topology. For $g$ runs through $\Gamma(F)$ the isomorphisms \eqref{e:isomorphisms}
induce an action of $\Gamma(F)$ on
$\mathcal{H}^i_v(A,F/F')$, which factors through an action of $\Gamma(F')$, and thus yield a $\Lambda(\Gamma(F'))$-module structure of
$\mathcal{H}^i_v(A,F/F')$.
In general, set
\begin{equation*}
\mathcal{H}^i_v(A,N/N'):=\varinjlim_{F\subset N} \mathcal{H}^i_v(A,F/F'),
\end{equation*}
and denote
$$\mathcal{W}_v^i:={\mathcal{H}_v^i}(A, L/L')^{\vee}.$$
Also, for each place $w$ of $L'$
sitting over $v$, denote $\mathcal{H}_w^i=\coh^i(\Psi_w,A(L_w))$ and
$\mathcal{W}_w^i={\mathcal{H}_w^i}^{\vee}$. 



\begin{definition}\label{d:vartheta} Define
$\le \vartheta_w^{(i)}:=\chi_{\Lambda(\Gamma'_w)}(\mathcal{W}_w^i)$ and $\le \vartheta_v^{(i)}:=\chi_{\Lambda(\Gamma')}(\mathcal{W}_v^i)$,
if $\mathcal{W}_w^i$ and $\mathcal{W}_v^i$ are finitely generated over the corresponding Iwasawa algebras.
\end{definition}
In this section, we give explicit expressions of $\le \vartheta_w^{(i)}$ and $\le \vartheta_v^{(i)}$, for $i=1,2$.

\begin{subsection}{General facts}\label{su:gf} In general, for an intermediate extension $N$ of $L/K$ and a place $\textsl{w}\mid v$ of $N'$,
write $\Lambda_{N'}=\Z_p[[\Gamma(N')]]$ and
$\Lambda_{N'_{\textsl{w}}}=Z_p[[\Gamma(N')_{\textsl{w}}]]$.
Let $F$ be a finite intermediate extension of $L/K$. By choosing a place $\textsf{w}_0\mid v$ of $F'$, one can actually make the identification
$$\mathcal{H}^i_v(A,F/F')=\Hom_{\Lambda_{F'_{\textsf{w}_0}}}(\Lambda_{F'}, \coh^i(\Psi(F)_{\textsf{w}_0}, A(F_{\textsf{w}_0}))),$$
via the assignment $\xi\mapsto f_{\xi}$ such that for
$\xi=(\xi_\textsf{w})_{\textsf{w}\mid v}\in \mathcal{H}^i_v(A,F/F')$, $f_{\xi}(g)={}^g\xi_{{g^{-1}}(\textsf{w}_0)}$, for
$g\in\Gamma(F')$.
Then it follows (see \cite[II.4.1, Proposition 1(b)]{bou70}) that the Pontryagin dual
\begin{equation}\label{e:hinv}
\mathcal{H}^i_v(A,F/F')^{\vee}=\Lambda_{F'}\otimes_{\Lambda_{F'_{\textsf{w}_0}}} \coh^i(\Psi(F)_{\textsf{w}_0}, A(F_{\textsf{w}_0}))^{\vee}.
\end{equation}
Now choose a place $w$ of $L'$ sitting over $v$ and, for every $F$, choose $\textsf{w}_0$ to be the place of $F'$ sitting below $w$.
Since $\mathcal{H}^i_v=\varinjlim_{F}\mathcal{H}^i_v(A,F/F')$, the duality and (\ref{e:hinv}) imply
\begin{equation}\label{e:winv}
\mathcal{W}^i_v=\varprojlim_{F} \Lambda_{F'}\otimes_{\Lambda_{F'_{\textsf{w}_0}}} \coh^i(\Psi(F)_{\textsf{w}_0}, A(F_{\textsf{w}_0}))^{\vee}.
\end{equation}
\begin{lemma}\label{l:gf1} If $\mathcal{W}_w^i$ is finitely generated over $\Lambda(\Gamma'_w)$, then
$$\mathcal{W}^i_v=\Lambda(\Gamma')\otimes_{\Lambda(\Gamma_{w}')} \mathcal{W}^i_{w},$$
and hence
$$\vartheta_v^{(i)}=\Lambda(\Gamma')\cdot \vartheta_w^{(i)}.$$
\end{lemma}
\begin{proof}
For an intermediate extension $N$, let $\textsl{w}$ be the place of $N'$ below $w$. Then
\begin{equation}\label{e:nf}
\varinjlim_{F\subset N}  \coh^i(\Psi(F)_{\textsf{w}_0}, A(F_{\textsf{w}_0}))=\coh^i(\Psi(N)_{\textsl{w}},A(N_{\textsl{w}})).
\end{equation}
Let $\mathcal{N}$ denote the family of intermediate extensions of $L/K$ satisfying the conditions: (a) The decomposition subgroup $\Gamma(N')_v$
is open in $\Gamma(N')$, and 
(b) the natural map $\Gamma'\longrightarrow \Gamma(N')$ (resp. $\Gamma\longrightarrow \Gamma(N)$)
induces an isomorphism between the decomposition subgroups. Suppose $N\in \mathcal{N}$.
By (a), the index $d_N:=[\Gamma(N'):\Gamma(N')_{\textsl{w}}]$ is finite, and hence
$\Lambda_{N'}$ is a free
$\Lambda_{N'_{\textsl{w}}}$-module of rank $=d_N$. Also, the equality \eqref{e:nf} still holds,
if the limit is taken only over $F$ satisfying the condition that $\Gal(N'/F')_{\textsl{w}}=\Gal(N'/F')$.
For such $F$ the index
$[\Gamma(F'):\Gamma(F')_{\textsf{w}_0}]=d_N$,
and hence the rank of the free $\Lambda_{F'_{\textsf{w}}}$-module $\Lambda_{F'}$ equals $d_N$, too. Therefore, the limit
\begin{equation}\label{e:nf2}
\varprojlim_{F} \Lambda_{F'}\otimes_{\Lambda_{F'_{\textsf{w}_0}}} \coh^i(\Psi(F)_{\textsf{w}_0}, A(F_{\textsf{w}_0}))^{\vee}
=\Lambda_{N'}\otimes_{\Lambda_{N'_{\textsl{w}}}}\coh^i(\Psi(N)_{\textsl{w}},A(N_{\textsl{w}}))^\vee.
\end{equation}
By (b), We can identify $\Gamma(N')_{\textsl{w}}$, $\Psi(N)_{\textsl{w}}$, and $N_{\textsl{w}}$
with $\Gamma'_w$, $\Psi_w$ and $L_w$, respectively.
Consequently, we can write
$\coh^i(\Psi(N)_{\textsl{w}},A(N_{\textsl{w}}))=\coh^i(\Psi_w,A(L_w))$ as a module over
$\Lambda_{N'_{\textsl{w}}}=\Lambda(\Gamma'_w)$. Thus, by \eqref{e:winv} and \eqref{e:nf2},
$$\mathcal{W}^i_v=\varprojlim_{N\in\mathcal{N}} \Lambda_{N'}\otimes_{\Lambda(\Gamma'_w)}\mathcal{W}_w^i.$$
Then the lemma follows, if the finitely generated module $\mathcal{W}_w^i$ is free over $\Lambda(\Gamma'_w)$.
In general, for a finitely generated $\Lambda(\Gamma'_w)$-module $\mathcal{W}$, there is an exact sequence
$$0\longrightarrow \mathcal{Y}\longrightarrow \Lambda(\Gamma'_w)^r\longrightarrow  \mathcal{W}\longrightarrow 0,$$
and hence the exact sequence (as $\Lambda_{N'}$ is always free over $\Lambda(\Gamma'_w)$)
$$0\longrightarrow \Lambda_{N'}\otimes_{\Lambda(\Gamma'_w)}\mathcal{Y}\longrightarrow \Lambda_{N'}^r\longrightarrow  \Lambda_{N'}\otimes_{\Lambda(\Gamma'_w)}\mathcal{W}\longrightarrow 0.$$
Since the system $\{\Lambda_{N'}\otimes_{\Lambda(\Gamma'_w)}\mathcal{Y}\}_{N\in\mathcal{N}}$ satisfies the Mittag-Leffler
condition, the canonical map $\Lambda(\Gamma')^r=\varprojlim_{N\in\mathcal{N}}\Lambda_{N'}^r\longrightarrow  \varprojlim_{N\in\mathcal{N}}\Lambda_{N'}\otimes_{\Lambda(\Gamma'_w)}\mathcal{W}$
is surjective. It follows from the commutative diagram of exact sequences
$$
\xymatrix{{} & \Lambda(\Gamma')\otimes_{\Lambda(\Gamma'_w)} \mathcal{Y} \ar[d]^j \ar[r] & \Lambda(\Gamma')^r\ar@{=}[d]  \ar@{->>}[r] &
\Lambda(\Gamma')\otimes_{\Lambda(\Gamma'_w)} \mathcal{W} \ar[d]^i\\
0\ar[r] & \varprojlim_{N\in\mathcal{N}}\Lambda_{N'} \otimes_{\Lambda(\Gamma'_w)}\mathcal{Y} \ar[r] &
\varprojlim_ {N\in\mathcal{N}}\Lambda_{N'}^r \ar@{->>}[r] &
\varprojlim_{N\in\mathcal{N}}\Lambda_{N'} \otimes_{\Lambda(\Gamma'_w)}\mathcal{W}}
$$
that the canonical map $i$ is surjective. Now, since $\mathcal{Y}$ is finitely generated over $\Lambda(\Gamma'_w)$ as well, the map $j$
is also surjective. Then the diagram shows $i$ is an isomorphism.
\end{proof}
The following setting is useful for computing $\mathcal{W}_w$.
Let $L''/K$ be an intermediate $\Z_p$-extension of $L/K$ so that
$L=L'L''$ and $K=L'\cap L''$. For each finite intermediate extension
$F/K$, let $F''_n$ denote the $n$th layer of the $\Z_p$-extension
$F'':=L''F$ over $F$. For simplicity, write $F_w$, ${F''_n}_w$ and $F''_w$ for the topological closure in $L_w$
of $F$, $F''_n$ and $F''$.
\begin{lemma}\label{l:gf2} Let the notation be as above. Then
$$\coh^i(\Psi_w,A({L}_w))=\varinjlim_{F\subset L'}
\coh^i(\Gal({F''}_w/F_w),A({F''}_w)).$$
In particular, if $\coh^1(\Gal({F''}_w/F_w),A({F''}_w))$ is finite, for all intermediate extension
$F/K$ of $L'/K$, then $\mathcal{W}_v^2=0$.
\end{lemma}
\begin{proof}
The first assertion is obvious.
If $\coh^1(\Gal({F''}_w/F_w),A({F''}_w))$ is finite, then the standard Herbrand quotient computation shows
$$|\coh^2(\Gal({F''_n}_w/F_w),A({F''_n}_w))|=|\coh^1(\Gal({F''_n}_w/F_w),A({F''_n}_w))|,$$
and hence is bounded
as $n\rightarrow \infty$.
Then the diagram (\ref{e:inf2}) implies
$$\coh^2(\Gal({F''}_w/F_w),A({F''}_w))=\varinjlim_{n} \coh^2(\Gal({F''_n}_w/F_w),A({F''_n}_w))=0.$$
Therefore, $\mathcal{W}_w^2=0$, whence $\mathcal{W}_v^2=0$, by Lemma \ref{l:gf1}.
\end{proof}

\end{subsection}

\begin{subsection}{The unramified case}\label{sub:loccoh}
Let $\Pi_v$ and $\pi_v$ be as in Definition \ref{d:badunram}.
\begin{lemma}\label{l:varunramified}
Let $K^{up}_v/K_{v}$ be the unique unramified $\Z_p$-extension.
Then
$$|\coh^1(\Gal(K^{up}_v/K_{v}),A(K^{up}_v))|=|\Z_p/\pi_v|.$$
In particular, if $A$ has good reduction at $v$, then
$\coh^1(\Gal(K^{up}_v/K_{v}),A(K^{up}_v))=0$.
\end{lemma}


\begin{proof}
By Proposition I.3.8, \cite{mil86}, the reduction map induces
$$\coh^1(\Gal(K^{un}_{v}/K_{v}),A(K^{un}_{v}))
=\coh^1(\Gal({\bar \F}_{v}/\F_{v}),\Pi_{v})),
$$
whence
$$\coh^1(\Gal(K^{up}_v/K_{v}),A(K^{up}_v))
=\coh^1(\Gal(K^{up}_v/K_{v}),\Pi_{v}^{\Gal(K_v^{un}/
K^{up}_v)}).
$$
Then we apply Corollary \ref{c:cofito}.
\end{proof}

\begin{proposition}\label{p:urwi}
Suppose $v\not\in S$. Then the following holds:
\begin{enumerate}
\item[(a)] If $A$ has good reduction at $v$, then
$\mathcal{W}_v^1=0$.
\item[(b)] $\vartheta_v^{(1)}=
\begin{cases}
(\pi_v), & \Psi_v\not=0 ;\\
(1), & \text{otherwise}.\\
\end{cases}$
\item[(c)] $\mathcal{W}_v^2=0$
\end{enumerate}
\end{proposition}
\begin{proof} The previous lemma implies (a), it also implies
(c), in view of Lemma \ref{l:gf2}.
Now $\mathcal{W}_w^1$ is trivial, if $\Psi_w=0$. On the other hand, if $\Psi_w\not=0$, then $\Gamma_w'=0$ and $\Psi_w\simeq \Z_p$, as
$L_w/K_v$ is unramified.
Then, (b) follows from  Lemma \ref{l:gf1} and Lemma \ref{l:varunramified}.
\end{proof}

\end{subsection}

\begin{subsection}{The good ordinary case}\label{sub:go}
In this section, we assume that $A$ has good ordinary reduction at $v$.
Let $\textsf{f}_{L',v}\in \Lambda(\Gamma')$ be as in Definition \ref{d:goodord}.
Our aim is the following:
\begin{proposition}\label{p:go}
Suppose $A$ has good ordinary reduction at $v\in S$. Then the following holds:
\begin{enumerate}
\item[(a)] If $L'/K$ is ramified at $v$, then $\W_v^1$ is pseudo-null over $\Lambda(\Gamma)$.
\item[(b)] If $L'/K$ is unramified at $v$, then $\vartheta_v^{(1)}=(\textsf{f}_{L',v})$.
\item[(c)] $\W_v^2=0$.
\end{enumerate}
\end{proposition}

\begin{proof}
Let the notation be as in Lemma \ref{l:gf2} and let $\Psi_w^1$ denote the inertia subgroup of $\Psi_w$.
If $L/L'$ is unramified at $v$, then by following the proof of Proposition \ref{p:urwi},
one can show that the proposition holds trivially with $\mathcal{W}_v^1=0$.

Thus, we assume that $L/L'$ is ramified at $v$. Let $F(n)$ denote the $n$th layer of $L'/K$ and let $L''$, $F(n)''$ be as in \S\ref{su:gf}.
Then the generalized Mazur's Theorem \cite[Theorem 2]{tan10} (see also \cite{cog96} for the number field case)
asserts that, for $m\geq n$,
we have the commutative diagram of exact sequences:
\begin{equation}\label{e:th2}
\xymatrix{
\Hom({\bar A}^t(\F_{{F(n)}_w})_p,\Q_p/\Z_p) \ar@{^{(}->}[r] \ar[d] & \coh^1(\Psi_w, A(F''(n)_w))\ar@{->>}[r] \ar[d]  & \Hom(\Psi_w^1,{\bar A}(\F_{F(n)_w}))\ar[d]\\
\Hom({\bar A}^t(\F_{{F(m)}_w})_p,\Q_p/\Z_p) \ar@{^{(}->}[r]  & \coh^1(\Psi_w, A(F''(m)_w))\ar@{->>}[r]   & \Hom(\Psi_w^1,{\bar A}(\F_{F(m)_w}))
}
\end{equation}
where the first down-arrow is induced by the norm map
$$\Nm_{F(m)_w/F(n)_w}:{\bar A}^t(\F_{F(m)_w})\longrightarrow {\bar A}^t(\F_{F(n)_w}).$$
In particular, $\coh^1(\Psi_w, A(F''(n)_w))$ is finite, for every $n$, and hence the assertion (c) follows from Lemma \ref{l:gf2}.
By taking the direct limit via the diagram \eqref{e:th2}, we deduce the exact sequence
\begin{equation}\label{e:awb}
0\longrightarrow \mathfrak{A}\longrightarrow   \W_w^1\longrightarrow \mathfrak{B}\longrightarrow 0,
\end{equation}
where $\mathfrak{A}$ is the Pontryagin dual of ${\bar A}(\F_{L_w'})_p$ and $\mathfrak{B}$ is the projective limit
$$\varprojlim_{m}{\bar A}^t(\F_{F(m)_w})_p$$
taking over the norm maps $\Nm_{F(m)_w/F(n)_w}$. It follows that $\W_v^1$ is finitely generated over $\Z_p$, and hence is
pseudo-null over $\Lambda(\Gamma')$, unless $\Gamma'_{v}\simeq \Z_p$, or $0$ .
We first consider the $\Gamma_{v}'= 0$ case, in which $L'_w=K_v$ and actually
$$|\W_w^1|=|{\bar A}^t(\F_v)_p|\cdot |{\bar A}(\F_v)_p|=|{\bar A}(\F_v)_p|^2.$$
It follows from Lemma \ref{l:gf1} that the $\Z_p$-ideal $\vartheta_v^{(1)}=(|{\bar A}(\F_v)|^2)$.
On the other hand, since $[v]_{L'/K}=id$ and all $\alpha_i$ and $1-q_v/\alpha_i$, $i=1,...,g$, are units in the maximal order $\O$ of the field $\Q_p(\alpha_1,....,\alpha_g)$, we have
$$(\textsf{f}_{L',v})=\prod_{i=1}^g(1-\alpha_i)^2\cdot \Z_p=\prod_{i=1}^g(1-\alpha_i)^2\cdot \prod_{i=1}^g(1-q_v/\alpha_i)^2\cdot\Z_p,$$
which, according to Mazur \cite[Corollary 4.3.7]{maz} (see also Lemma \ref{p:torcomp} and Remark \ref{r:tor}), equals to the ideal $(|{\bar A}(\F_v)|^2)$. Hence (b) is proved for the $\Gamma_v'=0$ case.

Suppose $\Gamma'_{v}\simeq \Z_p$. If
$L'/K$ is ramified at $v$, then $L'_{w}$ has finite residue field, and hence $\W_w^1$ is finite (by (\ref{e:awb})).
It follows that  $\W_w^1$ is pseudo-null over
$\Lambda(\Gamma'_{w})$, and the assertion (a) is proved, by Lemma \ref{l:gf1}.
Finally, in this case (b) follows from Proposition \ref{p:torcomp2}.
\end{proof}

\end{subsection}
\begin{subsection}{The split multiplicative case}\label{sub:sm}
In this section, we assume that $A$ has split multiplicative
reduction at $v$.
Let $\mathfrak{w}_v$ be the $\Z_p$-ideal in Definition \ref{d:splitmul}.
\begin{proposition}\label{p:spm}
Suppose $v\in S$ and $A$ has split multiplicative at $v$.
Then the following holds:
\begin{enumerate}
\item[(a)] If $\Gamma_v'=0$ and $\mathfrak{w}_v=0$, then
$\vartheta_v^{(2)}=0$; otherwise, $\mathcal{W}_v^2=0$ and
$\vartheta_v^{(2)}=(1)$.
\item[(b)] If $\Gamma_v'=0$, then $\vartheta_v^{(1)}=
\Lambda(\Gamma')\cdot\mathfrak{w}_v$.
\item[(c)] If $\Gamma_v'\simeq \Z_p$ with $\sigma$ a topological generator and $\Psi_v\simeq \Z_p$, then
$\vartheta_v^{(1)}=(\sigma-1)^g$.
\item[(d)] If $\Gamma_v'\simeq\Z_p$ and $\Psi_v=0$ or $\Gamma_v'\simeq
\Z_p^e$, $e>1$, then $\vartheta_v^{(1)}=(1)$.
\end{enumerate}
\end{proposition}

\begin{proof}
Since $\Omega\simeq \Z^g$ as $\Gamma$-modules,
$$\coh^3(\Psi_w/\Psi_w^{p^n},\Omega)\simeq
\coh^1(\Psi_w/\Psi_w^{p^n},\Omega)\simeq \Hom(\Psi_w/\Psi_w^{p^n},\Z^g)=0,$$
and hence
$\coh^3(\Psi_w,\Omega)=0.$
Also, Hilbert's theorem 90 implies
$\coh^1(\Psi_w,L_w^{\times})=0.$
Therefore, from the exact sequence of Galois-modules
$$0\longrightarrow \Omega\longrightarrow (L_w^{\times})^g\longrightarrow A(L_w)
\longrightarrow 0,$$
we deduce the long exact sequence:
\begin{equation}\label{e:period}
\xymatrix{ \coh^1(\Psi_w,A(L_w)) \ar@{^{(}->}[r] &
\coh^2(\Psi_w,\Omega)\ar[r] & \coh^2(\Psi_w,L_w^{\times})^g
\ar@{->>}[r] & \coh^2(\Psi_w,A(L_w))}.
\end{equation}
Consider the case where $\Gamma_v'\not=0$, and hence $\simeq\Z_p^e$, for some $e\geq 1$.
Let $L''$ be as in \S\ref{su:gf}, let ${L'_w}_n$ denote the $n$th layer of $L'_w/K_v$ and write ${L''_w}_n=L''{L'_w}_n$. Then it follows from the commutative diagram
$$
\xymatrix{\coh^2({L''_w}_n/{L'_w}_n,\le {L''_w}_n^{\times})\ar@{^{(}->}[r]\ar[d] & \coh^2({L'_w}_n, \overline{K}_v^{\times})_p\ar[r]_-{\sim}^-{inv}\ar[d] & \Q_p/\Z_p\ar[d]^{p^{e(m-n)}}\\
\coh^2({L''_w}_m/{L'_w}_m,\le {L''_w}_m^{\times})\ar@{^{(}->}[r] & \coh^2({L'_w}_m, \overline{K}_v^{\times})_p\ar[r]_-{\sim}^-{inv} & \Q_p/\Z_p},
$$
where the isomorphisms are local invariant maps of Brauer groups and the right vertical arrow is the multiplication by $p^{e(m-n)}$,
that $\coh^2(\Psi_w,\le L_w^{\times})=\varinjlim_{n} \coh^2({L''_w}_n/{L'_w}_n,\le {L''_w}_n^{\times})=0$.
Thus, (\ref{e:period}) implies $\W_w^2=0$ as well as the isomorphisms of
$\Lambda(\Gamma_w')$-modules:
$$\coh^1(\Psi_w,A(L_w))\simeq
\coh^2(\Psi_w,\Omega)\simeq \Hom(\Psi_w,\Q/\Z)^g.$$
Consequently, $\W_w^1\simeq\Z_p^g$, if $\Psi_v\simeq\Z_p$;
$\W_w^1=0$, if $\Psi_v=0$. Thus, (c), (d) and a part of (a) are
proved (by Lemma \ref{l:gf1}).

Now consider the $\Gamma'_v=0$ case, in which $L_w'=K_v$ and $\Psi_w=\Gamma_v\simeq\Z_p$, since $v\in S$.
Let ${L_w}_n$ denote the $n$th layer of $L_w/K_v$ with $\Gal({L_w}_n/K_v)=\Psi_w/\Psi_w^{p^n}$
and identify $K_v^{\times}/\Nm_{L_w/K_v}(L_w^{\times})\simeq \Gamma_v$ (via the reciprocity law).
Then, as $n\rightarrow\infty$, the commutative diagram
$$
\xymatrix{\coh^2(\Psi_w/\Psi_w^{p^n},\Omega) \ar[r]\ar[d]^{\simeq} & \coh^2(\Psi_w/\Psi_w^{p^n},({L_w}_n^{\times}))^g\ar[d]^{\simeq}\\
\Omega/p^n\Omega \ar[r] &   (\Gamma_v/\Gamma_v^{p^n})^g},
$$
tends to (by (\ref{e:inf2}))
\begin{equation}\label{e:omegal}
\xymatrix{
\coh^2(\Psi_w,\Omega) \ar[r]\ar[d]^{\simeq}   & \coh^2(\Psi_w,L_w^{\times})^g\ar[d]^{\simeq}\\
\Q_p/\Z_p\otimes \Omega \ar[r]^{\bar{\mathcal{R}}_v} & \Q_p/\Z_p\otimes \Gamma_v^g,}
\end{equation}
where $\bar{\mathcal{R}}_v$ is induced by the map in (\ref{e:mcrv}).
This implies that both
$\coh^2(\Psi_w,\Omega)$ and $\coh^2(\Psi_w,L_w^{\times})^g$ are of corank
$g$ over $\Z_p$. Then by (\ref{e:period}), $\coh^1(\Psi_w,A(L_w))$ and $\coh^2(\Psi_w,A(L_w))$ are of the same $\Z_p$-corank that equals the
$\Z_p$-rank
of both $\ker\be\left[\e {\mathcal{R}}_v\right]$ and $\coker\be\left[\e {\mathcal{R}}_v\right]$.

Suppose $\mathfrak{w}_v=0$. Then both $\W_w^1$
and $\W_w^2$ are infinite. By Lemma \ref{l:gf1},
$\vartheta_v^{(1)}=\vartheta_v^{(2)}=0$.
This proves the first part of (a) and a part of (b).

Suppose $\mathfrak{w}_v\not=0$. Then $\coker\be\left[\e {\mathcal{R}}_v\right]$ is finite and hence
$\coh^2(\Psi_w,\Omega)\longrightarrow \coh^2(\Psi_w,L_w^{\times})^g$ is surjective (as both groups are $p$-divisible)
with finite kernel.
In particular, $\mathcal{W}_w^2=0$ and $\vartheta_v^{(2)}=(1)$. Thus, the proof of (a) is
completed. To finish the proof of (b), we apply the snake lemma to the diagram
$$
\xymatrix{
0\ar[r] & \Omega \ar[r]\ar[d]^{\mathcal{R}_v} & \Q_p\otimes \Omega \ar[r] \ar[d] & \Q_p/\Z_p\otimes \Omega \ar[r]\ar[d]^{\bar{\mathcal{R}}_v} &  0\\
0\ar[r]  & \Gamma_v^g \ar[r] & \Q_p\otimes \Gamma_v^g \ar[r] & \Q_p/\Z_p\otimes \Gamma_v^g\ar[r] &  0}
$$
to show that $\coker\be\left[\e {\mathcal{R}}_v\right]\simeq\ker\be\left[\e\coh^2(\Psi_w,\Omega)\longrightarrow \coh^2(\Psi_w,L_w^{\times})^g\right] =\coh^1(\Psi_w,A(L_w))$.
\end{proof}

\begin{lemma}\label{l:zpsplit}
Suppose $L/K$ is a $\Z_p^d$-extension and $v\in S$ is a split-multiplicative place. Then $\coh^1(\Gamma_v, A(L_v))$ is cofinitely
generated over $\Z_p$ and
$${\corank}_{\Z_p} \coh^1(\Gamma_v, A(L_v))=\rank_{\Z_p}\coker\be\left[\e \mathcal{R}_v\right]\geq g(\rank_{\Z_p}\Gamma_v-1).$$
\end{lemma}
\begin{proof} The inequality follows from the fact that $\Z_p\otimes_{\Z}\Omega_v$ is of rank $g$.
Also, since $A$ and $A^t$ are isogenous, it is enough to show ${\corank}_{\Z_p} \coh^1(\Gamma_v, A^t(L_v))=\rank_{\Z_p}\coker\be\left[\e \mathcal{R}_v\right]$. By Lemma \ref{l:tatenorm}, ${\corank}_{\Z_p} \coh^1(\Gamma_v, A^t(L_v))=\rank_{\Z_p} A(K_v)/\Nm_{L_v/K_v}(A(L_v))$.
By Class Field Theory
$$A(K_v)/\Nm_{L_v/K_v}(A(L_v))=(K_v^{\times})^g/\overline{\Omega}\cdot (\Nm_{L_v/K_v}(L_v))^g\simeq \coker\be\left[\e \mathcal{R}_v\right].$$
Here $\overline{\Omega}$ denotes the topological closure of $\Omega$.
\end{proof}
\end{subsection}

\begin{subsection}{The proofs of Proposition \ref{p:iwasawa}}\label{su:tiw}
For an intermediate extension $N$ of $L/K$, let
$$
\mathcal{L}_N:\coh^1(N,A[p^{\infty}])\longrightarrow
\mathcal{H}^1(A/N)
$$
denote the direct limit of the localization maps $\mathcal{L}_F$ (see \eqref{e:loc}) for $F$ running through
finite intermediate extension of $N/K$, with $\mathcal{H}^1(A/N)=\varinjlim_{F} \mathcal{H}^1(A/F)$. 

Note that a version of Nakayama's lemma (see \cite[p.279]{w82}) implies the following:
\begin{lemma}\label{l:equiv} $X_L$ is finitely generated over $\Lambda(\Gamma)$
if and only if $\Sel_{p^{\infty}}(A/L)^{\Gamma}$ is cofinitely generated over $\Z_p$.
\end{lemma}
\begin{proof} (of Proposition \ref{p:iwasawa})
Let $\mathcal{P}$ stand for the statement that $X_L$
is finitely generated over $\Lambda(\Gamma)$.
Consider the commutative diagram:
\begin{equation*}\label{e:reslk}
\xymatrix{
{} & \coh^1(K,A[p^{\infty}])\ar[r]^{\res_{L/K}}\ar[d]^{\mathcal{L}_K}  &
\coh^1(L,A[p^{\infty}])^{\Gamma}\ar[d]^{\mathcal{L}_L}\\
\mathbf{B}:=\bigoplus_{\text{all}\;v}\coh^1(\Gamma_v,A(L_v)) \ar@{^{(}->}[r]  & \mathcal{H}^1(A/K) \ar[r]
& \mathcal{H}^1(A/L).}
\end{equation*}
By \eqref{e:kerreslf} and \eqref{e:cokerreslf}
both $\ker \be\left[\e \res_{L/K}\right]$ and $\coker \be\left[\e \res_{L/K}\right]$
are finite.
Since $\Sel_{p^{\infty}}(K)$
is cofinitely generated over $\Z_p$, in view of Lemma \ref{l:equiv},
we see that $\mathcal{P}$ is equivalent to the condition that
the intersection $\image(\mathcal{L}_K)\cap \mathbf{B}$
is cofinitely generated over $\Z_p$.
Now, Theorem \ref{t:gct} implies that $\coker\be\left[\e\mathcal{L}_K\right]$
is cofinitely generated over $\Z_p$. Therefore, $\mathcal{P}$ holds
if and only if $\mathbf{B}$ is also cofinitely generated over $\Z_p$. But, we already know from
Lemma \ref{l:varunramified} that $\bigoplus_{v\not\in S}\coh^1(\Gamma_v,A(L_v))$ is finite.
\end{proof}

\end{subsection}

\begin{subsection}{Deeply ramified extensions}\label{su:deep}
In this section, we assume that $K$ is of characteristic $p$
and $A$ has good reduction at a given place $v$.
Let $\widehat{A}$ denote the associated formal group
obtained via the formal completion of $A$ along the zero section of the
N$\acute{\text{e}}$ron model. Since the N$\acute{\text{e}}$ron model is stable over
$\overline{K}_v$, it makes sense to consider the cohomology group
$\coh^1(K_v,\widehat{A}(\O_{\overline{K}_v}))$.
For simplicity, we write $\coh^1(K_v,\widehat{A})$ for it.

\begin{mytheorem}\label{t:ss}
Suppose $char.(K)=p$ and the reduction ${\bar{A}}$ of $A$ at  a place $v$ ramified over
$L/K$ satisfies ${\bar{A}}[p^{\infty}]({\overline{\F}}_v)=0$. Then both the natural maps
$$\coh^1(\Gamma_v, \widehat{A}(\O_{L_v}))\longrightarrow \coh^1(K_v,\widehat{A})
\longrightarrow \coh^1(K_v,A)_p$$
are isomorphisms. In particular, the cohomology group $\coh^1(\Gamma_v, A(L_v))$
is of {\em{(}}countably{\em{)}} infinite corank over $\Z_p$.
\end{mytheorem}
The proof, given below, is basically contained in Coates and Greenberg \cite{cog96},
which is written under the assumption
that $K_v$ is a finite extension of $\Q_p$.
In order to apply their result to our situation, we follow the paper step by step
to conclude that all material contained in its \S2 and \S3, which
is sufficient for proving the above theorem,
actually remain valid in the characteristic $p$ case, with only two exceptions:
\vskip5pt
\noindent
(A) Theorem 2.13, \cite{cog96}, implies that every ramified $\Z_p^d$-extension over $K_v$
is \textit{deeply ramified} (see p.143, [op.cit.]).
Its proof is based on Sen \cite{sen72}
using the existence of \textit{non-small} abelian subquotients of the Galois group (see Proposition 3.3, [op.cit.]),
which does not hold in our case. This gap is fixed by Lemma \ref{l:deep}
below.
\vskip5pt
\noindent
(B) Proposition 2.5, \cite{cog96}, is proved by using
cyclotomic extensions to construct a $\Z_p$-extension $\Phi/K_v$ so that
if $\Phi_t$ is the $t$'th layer 
then
$$\Tr_{\Phi_t/K_v}(m_{\Phi_t})\subset m_{K_v}^{n(t)},$$
where  $n(t)$ is an integer valued function of
$t$ so that $n(t)\rightarrow \infty$ as $t\rightarrow\infty$.
In order to have a proof of the above proposition in the characteristic $p$ case, we only
need to find a suitable $\Z_p$-extension $\Phi/K_v$ satisfying the above condition.
In view of Lemma 2.3, [op.cit], we see that it is enough to have $\Phi/K_v$ so that
\begin{equation}\label{e:Phi}
\ord_{\Phi_t}(\delta(\Phi_t/K_v))/e(\Phi_t/K_v)
\rightarrow\infty,\;\text{as}\; t\rightarrow\infty.
\end{equation}
Here, for a finite extension $\mathcal{F}/K_v$, $\delta(\mathcal{F}/K_v)$ and $e(\mathcal{F}/K_v)$ denote, respectively,
the different and the ramification index. Again, our
Lemma \ref{l:deep} below asserts that in the characteristic $p$ case, the condition
(\ref{e:Phi}) is satisfied, as long as $\Phi/K_v$ is ramified.

For a pro-finite abelian extension $L_v/K_v$ with Galois group $\Gamma_v$,
let $\Gamma_v^{(w)}$, for each $w\in [-1,\infty)$, denote the $w$'th ramification subgroup
in the upper numbering. The reciprocity law maps $U_w$ onto $\Gamma_v^{(w)}$
where $\{U_w\}$ is the usual filtration of the units of $K_v$ (see \cite[ XV.2]{ser79}).

\begin{lemma}\label{l:deep}
Suppose $char.(K)=p$ and $L_v/K_v$ is a pro-finite abelian extension with Galois group $\Gamma_v$.
Then the following holds:
\begin{enumerate}
\item[(a)] We have $(\Gamma_v^{(w)})^p\subset \Gamma_v^{(pw)}$ for every $w$.
In particular, if $\Gamma_v$ is finite, then it is small in the sense of \cite{sen72}.
\item[(b)] If $L_v/K_v$ is a ramified $\Z_p$-extension and ${L_v}_n$ is the $n$'th layer, then
$$\ord_{L_{v,n}}(\delta({L_v}_n/K_v))/e({L_v}_n/K_v)\rightarrow\infty,\;\text{as}\;
n\rightarrow\infty.$$
\item[(c)] If $L_v/K_v$ is a ramified $\Z_p^d$-extension, then it is deeply ramified in the
sense of \cite{cog96}.
\end{enumerate}
\end{lemma}
\begin{proof}
For each $x\in m_{K_v}$, we have $(1+x)^p=1+x^p$, and hence $(U_{w})^p\subset U_{pw}$.
This proves (a) by applying the reciprocity law.
Suppose $L_v/K_v$ is a $\Z_p$-extension and let $\pi$ be a prime element of $K_v$.
If $\chi$ is a continuous character of $\Gamma_v$ with conductor
$f(\chi)$, then (as the order of $\chi$ is a power of $p$) from the relation $\chi^p(g)=\chi(g^p)$, we see that the
conductor $f(\chi^p)$ satisfies
\begin{equation}\label{e:pcondp}
p\cdot f(\chi^p) \leq f(\chi).
\end{equation}

Let $\pi^{\Delta_n}\O_{K_v}$ denote the discriminant of the cyclic extension ${L_v}_n/K_v$.
Since the dual group of $\Gal({L_v}_n/K_v)$ is cyclic, the conductor-discriminant formula together with
(\ref{e:pcondp}) imply that as $n\rightarrow\infty$,
$$\Delta_n\geq C_1p^{2n}+O(p^{2n-1}), \;\text{for some positive constant}\; C_1. $$
Consequently, as $n\rightarrow\infty$,
$$\ord_{{L_v}_n}(\delta({L_v}_n/K_v))/e({L_v}_n/K_v)\geq C_2p^{n}+O(p^{n-1}),
\;\text{for some positive constant}\; C_2. $$
Then (b) is proved. Also, if $\ord$ is the normalized valuation on ${\overline K}_v$
with $\ord(\pi)=1$, then
we have
$$\ord(\delta({L_v}_n/K_v))\geq C_3p^{n}+O(p^{n-1}), \;\text{for some positive constant}\; C_3,$$
and hence $\ord(\delta({L_v}_n/K_v))\rightarrow\infty$ as $n\rightarrow\infty$.
This shows that $L_v/K_v$ is deeply ramified in the sense of \cite{cog96} and proves
(c) in the $d=1$ case. In general, we only need to note that there is a
ramified intermediate $\Z_p$-extension $L_v^{0}/K_v$ of $L_v/K_v$, and since
$L_v^{0}/K_v$ is deeply ramified, the multiplicity of the different implies that $L_v/K_v$
is also deeply ramified (see p.143, [op.cit.]).
\end{proof}

Thus, results in \cite{cog96}, \S2 and \S3 can be applied to our situation.

\begin{mytheorem}\label{t:cogp}
Suppose $K$ is a global field of characteristic $p$ and $v$ is a place of $K$.
If $\mathfrak{F}$ is a commutative formal group law
over $\O_{K_v}$ and $L/K$ is a $\Z_p^d$-extension ramified at $v$, then
$$
\coh^1(\Gamma_v, \mathfrak{F}(m_{L_v}))=\coh^1(\Gal({\overline{K}_v}/K_v),\mathfrak{F}(m_{\overline{K}_v})).
$$
\end{mytheorem}
\begin{proof}
Theorem 3.1, \cite{cog96}, together with the inflation-restriction exact sequence.
\end{proof}

\begin{proof} (of Theorem \ref{t:ss})
Theorem \ref{t:cogp} says $\coh^1(\Gamma_v, \widehat{A}(\O_{L_v}))
\longrightarrow \coh^1(K_v,\widehat{A})$
is an isomorphism. Since the torsion group
${\bar A}({\bar \F}_{K_v})$ contains no element of order $p$,
the long exact cohomology sequence
$$\dots\longrightarrow\coh^0(\F_{K_v},{\bar{A}})\longrightarrow \coh^1(K_v,\widehat{A})\longrightarrow
\coh^1(K_v,A)\longrightarrow \coh^1(\F_{K_v},{\bar{A}})\longrightarrow\dots$$
says the map
$ \coh^1(K_v,\widehat{A})\longrightarrow
\coh^1(K_v,A)_p$ is an isomorphism. Thus, in the commutative diagram
$$\xymatrix{\coh^1(\Gamma_v,\widehat{A}(\O_{L_v}))\ar[r]\ar[d] &  \coh^1(\Gamma_v,A(L_v))\ar[d]\\
\coh^1(K_v,\widehat{A})\ar[r] & \coh^1(K_v,A)_p}
$$
all arrows are isomorphisms.
Then the theorem is clear, since now $\coh^1(\Gamma_v,A(L_v))$ is dual to the $p$-completion of $A^t(K_v)$
which contains (the $p$-completion of) ${\widehat{A^t}}(\O_{K_v})$, a $\Z_p$-module
of infinite rank, \cite{vol95}.

\end{proof}

\end{subsection}

\end{section}



\begin{section}{The restriction and the localization}
Consider the restriction map
$\res_{L/L'}:\coh^1(L',A[p^{\infty}])\longrightarrow
\coh^1(L,A[p^{\infty}])^{\Psi}$ and
let $\mathcal{L}_N$ be the localization defined in \S\ref{su:tiw}.
In this section, we give explicit expression of the related objects, 
especially $\coker \be\left[\e \res_{L/L'}\right]$ 
and $\coker \be\left[\e \mathcal{L}_{L}\right]$, as well as the corresponding characteristic ideals.

\begin{subsection}{The map $\res_{L/L'}$}\label{sub:restriction} Obviously,
\begin{equation}\label{e:kernelres}
\ker \be\left[\e \res_{L/L'}\right]=\coh^1(\Psi, A[p^{\infty}](L)).
\end{equation}

\begin{proposition}\label{p:injection} We have $\coker \be\left[\e \res_{L/L'}\right]=0$.
\end{proposition}
\begin{proof} Since $\Psi\simeq\Z_p$ is of cohomology dimension $1$, $\coh^2(\Psi,A[p^{\infty}](L))=0$, whence
$$\coker \be\left[\e \res_{L/L'}\right]=\ker \be\left[\coh^2(\Psi,A[p^{\infty}](L))\longrightarrow
\coh^2(L',A[p^{\infty}])\right]=0.$$
\end{proof}

\end{subsection}

\begin{subsection}{The map $\mathcal{L}_L$}\label{sub:loc}
Since the map $\coh^1(F,A[p^{\infty}])\longrightarrow \coh^1(F,A)_p$ induced by the Kummer exact sequences
is surjective, Theorem \ref{t:gct} actually says
\begin{equation}\label{e:cokerLF}
\coker\be\left[\e \mathcal{L}_F\right]^{\vee}\simeq \T_p\Sel(A^t/F)
\end{equation}
identified as a topological subgroup of $\mathcal{H}^0(A^t/F)$.
By taking the projective limit over $F$, we deduce
\begin{equation}\label{e:tps}
\coker\be\left[\e \mathcal{L}_L\right]^{\vee}\simeq \varprojlim_F \T_p\Sel(A^t/F).
\end{equation}

\begin{proposition}\label{p:loc} Suppose $L/K$ is a $\Z_p^d$-extension and $X_L$ is torsion over
$\Lambda(\Gamma)$. Then for $d\geq 2$ $\coker\be\left[\e \mathcal{L}_L\right]^{\vee}=0$, and also
$$\coker\be\left[\e \mathcal{L}_L\right]^{\vee}=\begin{cases}
A^t[p^{\infty}](K), & \text{if}\; d=0;\\
\varprojlim_{F}A^t[p^{\infty}](F), & \text{if}\; d=1;\\
\end{cases} $$
where the projective limit is taken over norm maps between finite intermediate extensions of $L/K$.
\end{proposition}

\begin{proof} For simplicity, write $X^t_L=\Sel_{p^{\infty}}(A^t/L)^{\vee}$. Then, as $A$ and $A^t$ are isogenous,
$X^t_L$ is also torsion.
Note that in general, $A^t[p^{\infty}](F)$ is contained in $\T_p\Sel(A^t/F)$ as its torsion part and
$\rank_{\Z_p}\T_p\Sel(A^t/F)=\corank_{\Z_p}\Sel_{p^{\infty}}(A^t/F)$.
Thus, if $d=0$, then, since $X^t_K$ is torsion, $\Sel_{p^{\infty}}(A^t/K)$ is finite,
whence the proposition is a direct consequence of Theorem \ref{t:gct}.

Let $K_n$ denote the $n$th layer of $L/K$. Suppose $d=1$. Then the ${\Z_p}$-rank of $\T_p\Sel(A^t/K_n)$ is stable as
$n\rightarrow\infty$. Thus, for $m\geq n\gg 0$,
the map $\T_p\Sel(A^t/K_n)\longrightarrow \T_p\Sel(A^t/K_m)$ (induced by the restriction map) gives rise to an isomorphism between their
$\Z_p$-free parts. Since the projective limit in \eqref{e:tps}
is taken over the maps $\T_p\Sel(A^t/K_m)\longrightarrow \T_p\Sel(A^t/K_n)$ induced by the corestriction maps on Selmer groups,
the restriction-corestriction formula implies $\varprojlim_n \T_p\Sel(A^t/K_n)/A^t[p^{\infty}](K_n)$ vanishes.
Therefore, $\varprojlim_n \T_p\Sel(A^t/K_n)=\varprojlim_n A^t[p^{\infty}](K_n)$ as desired.

Then we prove by induction on $d$. It is sufficient to show that for each $n$, there exists some intermediate $\Z_p^{d-1}$-extension $L_1$ of
$L/K_n$ such that $\mathcal{L}_{L_1}$ is surjective. 
For this purpose, we first observe that as $\Lambda(\Gamma)$ is a finite $\Lambda(\Gal(L/K_n))$-module,
$X^t_L$ is also torsion over $\Lambda(\Gal(L/K_n)$. Thus, without loss of the generality, we may assume that $K_n=K$.

Since $\Gamma$ is commutative, a basis of the vector space $V=\overline{\Q}_p\otimes_{\Z_p} (A^t[p^{\infty}](L)_{div})^{\vee}$
can be found so that the action of every $\gamma\in\Gamma$ on $V$ is represented by
an upper triangular matrix in which the diagonal entries are characters contained in $\Hom_{cont} (\Gamma, \overline{\Q}_p^\times)$.
Since $A^t[p^{\infty}](F)$ is finite for every finite extension $F/K$,
every such character is of infinite order, and hence its kernel is isomorphic to $\Z_p^j$ for some $j\leq d-1$. Thus, the union of the kernels of these characters
is a proper subset $\Upsilon$ of $\Gamma$. Choose $L_1$ so that the intersection
$\Gal(L/L_1)\cap\Upsilon$ contains only the identity element of $\Gamma$. Then, for
every non-trivial $\gamma\in \Gal(L/L_1)$, $1-\gamma$ induces an invertible linear operator on $V$. It follows that
$A^t[p^{\infty}](L_1)$ is finite, and hence
$$\varprojlim_{F\subset L_1} A^t[p^{\infty}](F)=0.$$
Thus, if $X^t_{L_1}$ is torsion over
$\Lambda(\Gal(L_1/K))$, then $\coker\be\left[\e \mathcal{L}_{L_1}\right]=0$, by the induction hypothesis.
Let $\xi\in \Lambda(\Gamma)$ be a non-zero element such that $\xi\cdot \Sel_{p^{\infty}}(A^t/L)=0$. Since $\xi$ can only be divisible by
finitely many non-associated elements of the form $\gamma-1$, $\gamma\in\Gamma$, we can choose
$L_1$ so that ${\bar{\xi}}:=p_{L/L_1}(\xi)\not=0$. Consequently, ${\bar{\xi}}\cdot \Sel_{p^{\infty}}(A^t/L_1)$ is in the kernel
of $\res_{L/L_1}$, and is actually finite by \eqref{e:kernelres} and Lemma \ref{l:cofito}. This implies $X^t_{L_1}$ is torsion.
\end{proof}

\end{subsection}

\begin{subsection}{The operator
$\psi-1$}\label{sub:mpsi}
Fix an intermediate $\Z_p$-extension $L''/K$
so that $L=L'L''$, $K=L'\cap L''$. The restriction of Galois action
induces the isomorphism
$\Psi\stackrel{\sim}{\longrightarrow}
\Gal(L''/K)$.
Let $L''_n$ denote the $n$th layer of $L''/K$,
and, for each finite intermediate extension $F/K$ of $L/K$,
denote $F_{n}''=FL''_n$. For positive integers $m$, $n$, let
$$\ker_{n,m}(F):=\ker\be\left[\e
\xymatrix{\coh^2(F,A[p^m])\ar[r]^-{res} &
\coh^2(F_{n}'',A[p^m])}\right].$$
\begin{lemma}\label{l:m1m2}
The natural map
$$ \xymatrix{ \varinjlim_{m} \coh^2(F,A[p^m]) \ar[r] & \coh^2(F,A)_p }$$ is injective, and consequently,
$$\varinjlim_{m} \le \ker_{n,m}(F)=0.$$
\end{lemma}
\begin{proof} Consider the commutative diagram of exact sequences
\begin{equation*}
\xymatrix{
\coh^1(F,A)/p^{m_2}\coh^1(F,A) \ar@{^{(}->}[r] \ar[d]^{p^{m_1-m_2}}  & \coh^2(F,A[p^{m_2}]) \ar[r]^-{i_{m_2}}\ar[d] &
\coh^2(F,A)_p \ar@{=}[d]\\
\coh^1(F,A)/p^{m_1}\coh^1(F,A) \ar@{^{(}->}[r]   & \coh^2(F,A[p^{m_1}]) \ar[r]^-{i_{m_1}} &
\coh^2(F,A)_p, }
\end{equation*}
 where the left vertical arrow is induced by the multiplication by
$p^{m_1-m_2}$. If an element $x\in \ker \be\left[\e i_{m_2} \right]$ is represented by some $y\in \coh^1(F,A)_p$ of
order $p^r$ and $m_1\geq m_2+r$, then $x$ vanishes under $ \coh^2(F,A[p^{m_2}]) \longrightarrow
\coh^2(F,A[p^{m_1}])$. This proves the first assertion. Then we observe that $\varinjlim_{m} \le \ker_{n,m}(F)=\ker \left[\e   \varinjlim_{m}\coh^2(F,A[p^m]) \longrightarrow  \varinjlim_{m}\coh^2(F_n'',A[p^m])  \right]$,
and by Proposition \ref{p:coh2st} we have the commutative diagram
$$\xymatrix{ \varinjlim_{m} \coh^2(F,A[p^m]) \ar@{^{(}->}[r]\ar[d] &   \coh^2(F,A)_p \ar@{^{(}->}[r]\ar[d] &
\mathcal{H}^2(A/F)\ar[d] \\
\varinjlim_{m} \coh^2(F_n'',A[p^m]) \ar@{^{(}->}[r] &  \coh^2(F_n'' ,A)_p \ar@{^{(}->}[r] &
\mathcal{H}^2(A/F_n'').}
$$
The right vertical arrow is also an injection, as every real place of $K$ splits completely over $L$.
\end{proof}

\begin{proposition}\label{p:psi-1}
We have
$$\coh^1(L,A[p^{\infty}])=
(\psi-1)\coh^1(L,A[p^{\infty}]).$$
\end{proposition}
\begin{proof}
Suppose $x\in\coh^1(L,A[p^{\infty}])$ is obtained from some
$x_k\in \coh^1(F_k'',A[p^{m}])$ for some finite intermediate
field $F/K$ of $L'/K$ and some integers $k$ and $m$.
Choose $n\geq m+k$ and let $x_n$ denote the image of $x_k$ under the
restriction map
$\coh^1(F_k'',A[p^{m}])\longrightarrow
\coh^1(F_n'',A[p^{m}])$. Then $x_n$ is annihilated by $p^m$ and fixed by $\Gal(F_n''/F_k'')$, whence contained
in the kernel of the norm map $\Nm_{\Gal(F_n''/F_k'')} $.
This implies $x_n\in\ker \be \left[ \le \Nm_{\Gal(F_n''/F)}\right]$.  If we choose the restriction of
$\psi$ to be the generator of $\Gal(F_n''/F)$, then $x_n$ determines a class ${\bar{x}}_n\in
\coh^1(F_n''/F,\coh^1(F_n'',A[p^m]))$.

In view of (\ref{e:spcoh3}) and (\ref{e:inf3}),
we see that if $n\geq 2m+k$,
then $d_{F_n''}^{1,1}(\bar{x}_n)=0$, and hence by (\ref{e:spcoh3}) again
$\bar{x}_n$ is contained in the image of
$\ker_{n,m}(F)$. Let $m\rightarrow\infty$.
Then Lemma \ref{l:m1m2}
says that, under $\coh^1(F_n'',A[p^{m}])\longrightarrow
\coh^1(F_n'',A[p^{\infty}])$, the image of $x_n$
is indeed contained in
$(\psi-1)\coh^1(F_n'',A[p^{\infty}])$.
\end{proof}

An application is in order. Let $ \mathcal{L}_L^{\Psi}: \coh^1(L,A[p^{\infty}])^{\Psi} \longrightarrow
 \mathcal{H}^1(A/L)^{\Psi}$ denote the restriction of $\mathcal{L}_L$ to $\coh^1(L,A[p^{\infty}])^{\Psi}$.
Consider the commutative diagram of exact sequences
$$
\xymatrix{
0 \ar[r] &  \Sel_{p^{\infty}}(A/L) \ar[r]\ar[d]^{\psi-1} &
\coh^1(L,A[p^{\infty}]) \ar[r]^{\mathcal{L}_L} \ar@{->>}[d]^{\psi-1}
& \mathcal{H}^1(A/L) \ar[r] \ar[d]^{\psi-1} & \coker\le\left[ \mathcal{L}_L \right] \ar[r]  \ar[d]^{\psi-1} & 0 \\
0 \ar[r] &  \Sel_{p^{\infty}}(A/L) \ar[r] &
\coh^1(L,A[p^{\infty}]) \ar[r]^{\mathcal{L}_L}
& \mathcal{H}^1(A/L) \ar[r]  & \coker\le\left[ \mathcal{L}_L \right] \ar[r]   & 0 .}
$$
By diagram chasing (snake lemma), we deduce the exact sequence
\begin{equation}\label{e:psipart}
\xymatrix{0 \ar[r] & \Sel_{p^{\infty}}(A/L)/(\psi-1)\Sel_{p^{\infty}}(A/L) \ar[r] & \coker\le\left[ \mathcal{L}_L^{\Psi}\right]
\ar[r] & (\coker\le\left[ \mathcal{L}_L \right])^{\Psi} \ar[r] & 0.}
\end{equation}

\end{subsection}

\begin{subsection}{A derived equality}\label{su:derived}
Recall the group $\mathcal{H}_v^i(A,L/L')$ defined in \S\ref{s:local}.
Write
$$\mathcal{H}^i(A,L/L')=\bigoplus_{v}\mathcal{H}_v^i(A,L/L')
\subset \mathcal{H}^i(A/L'),$$
where $v$ runs through all places of $K$, 
and set
$$\mathcal{W}^i(A,L/L')=\mathcal{H}^i(A,L/L')^{\vee}=\prod_{v}\mathcal{W}_v^i.$$
Consider the commutative diagram of exact sequences
\begin{equation}\label{e:diagramlong}
\xymatrix{\coh^1(\Psi,A[p^{\infty}]) \ar@{^{(}->}[r]\ar[d]^{\mathcal{L}_{L/L'}} & \coh^1(L',A[p^{\infty}]) \ar[r]^-{\res_{L/L'}}
\ar[d]^{\mathcal{L}_{L'}} & \coh^1(L,A[p^{\infty}])^{\Psi} \ar[r]\ar[d]^{\mathcal{L}_{L}^{\Psi}} & 0\ar[d]\\
\mathcal{H}^1(A,L/L')\ar@{^{(}->}[r] & \mathcal{H}^1(A/L') \ar[r]^{res^1} & \mathcal{H}^1(A/L)^{\Psi} \ar@{->>}[r] & \mathcal{H}^2(A,L/L').}
\end{equation}
In the diagram, the ``$0$'' term is due to Proposition \ref{p:injection}, while the surjection follows from the exact sequence
\eqref{e:infres} and Proposition \ref{p:coh2st}. From it, we derive the equality of alternating products:
\begin{equation}\label{e:derived}
\begin{array}{rl}
{} & \chi_{\Lambda(\Gamma')}(\ker\le\left[\mathcal{L}_{L/L'}\right]^{\vee})\cdot \chi_{\Lambda(\Gamma')}(\coker\le\left[\mathcal{L}_{L'}\right]^{\vee})
\cdot \chi_{\Lambda(\Gamma')}(\ker\le\left[\mathcal{L}_{L}^{\Psi}\right]^{\vee})\cdot \chi_{\Lambda(\Gamma')}(\mathcal{W}^2(A,L/L'))\\
= & \chi_{\Lambda(\Gamma')}(\coker\le\left[\mathcal{L}_{L/L'}\right]^{\vee})\cdot\chi_{\Lambda(\Gamma')}(\ker\le\left[\mathcal{L}_{L'}\right]^{\vee})
\cdot\chi_{\Lambda(\Gamma')}(\coker\le\left[ \mathcal{L}_{L}^{\Psi}\right]^{\vee})
\end{array}
\end{equation}
\begin{lemma}\label{l:derived} The $\Lambda(\Gamma')$-modules $\mathcal{W}^1(A,L/L')$ and $\mathcal{W}^2(A,L/L')$
are either both torsion or both non-torsion. If they are torsion, then
$$
\begin{array}{rl}
{} & \chi_{\Lambda(\Gamma')}( \ker\le\left[
\res_{L/L'}\right]^{\vee})\cdot \chi_{\Lambda(\Gamma')}(\coker\le\left[\mathcal{L}_{L'}\right]^{\vee})
\cdot \chi_{\Lambda(\Gamma')}(X_L/(\psi-1)X_L)\\
= & \prod_{v}\vartheta_v^{(1)}\cdot\Theta_{L'}
\cdot \chi_{\Lambda(\Gamma')}(X_L^{\Psi})\cdot
\chi_{\Lambda(\Gamma')}(\coker\le\left[ \mathcal{L}_L \right]^{\vee}/(\psi-1)\coker\le\left[ \mathcal{L}_L \right]^{\vee})
\end{array}
$$
\end{lemma}
\begin{proof}
By Propositions \ref{p:urwi}, \ref{p:go}, \ref{p:spm}, $\mathcal{W}_v^1(A,L/L')$ (resp. $\mathcal{W}_v^2(A,L/L')$) is non-torsion
if and only if $v\in S$ is a split-multiplicative place of $A$, $\Gamma_v'=0$ and $\mathfrak{w}_v=0$. This proves the first assertion.
Furthermore, if $\mathcal{W}_v^1(A,L/L')$ and $\mathcal{W}_v^2(A,L/L')$ are torsion, the it follows that $\chi_{\Lambda(\Gamma')}(\mathcal{W}^2(A,L/L')=(1)$
and $\chi_{\Lambda(\Gamma')}(\mathcal{W}^1(A.L/L'))=\prod_{v}\vartheta_v^{(1)}$.

Recall that $\coh^1(\Psi,A[p^\infty])^\vee$ is torsion by Corollary \ref{c:chicoh1}. Thus, if $\mathcal{W}^1(A, L/L')=\mathcal{H}^1(A,L/L')^\vee$ is also torsion then the left vertical arrow in \eqref{e:diagramlong}
yields
$$\chi_{\Lambda(\Gamma')}(\ker\le\left[\mathcal{L}_{L/L'}\right]^{\vee})\cdot \chi_{\Lambda(\Gamma')}(\mathcal{W}^1(A,L/L'))
=\chi_{\Lambda(\Gamma')}(\coker\le\left[\mathcal{L}_{L/L'}\right]^{\vee})\cdot \chi_{\Lambda(\Gamma')}( \ker\le\left[
\res_{L/L'}\right]^{\vee}).
$$
Therefor, in \eqref{e:derived}, the terms $\chi_{\Lambda(\Gamma')}(\ker\le\left[\mathcal{L}_{L/L'}\right]^{\vee})$ and $\chi_{\Lambda(\Gamma')}(\coker\le\left[\mathcal{L}_{L/L'}\right]^{\vee})$ can be replaced, respectively, by
$\chi_{\Lambda(\Gamma')}( \ker\le\left[
\res_{L/L'}\right]^{\vee})$ and $\chi_{\Lambda(\Gamma')}(\mathcal{W}^1(A,L/L'))$. Then, since
$\ker\le\left[\mathcal{L}_{L'}\right]^{\vee}=X_{L'}$ and
$\ker\le\left[\mathcal{L}_{L}^{\Psi}\right]^{\vee}=X_L/(\psi-1)X_L$, the lemma follow from the equality
$$\chi_{\Lambda(\Gamma')}(\coker\le\left[ \mathcal{L}_L^{\Psi}\right]^{\vee})
=\chi_{\Lambda(\Gamma')}(X_L^{\Psi})\cdot
\chi_{\Lambda(\Gamma')}(\coker\le\left[ \mathcal{L}_L \right]^{\vee}/(\psi-1)\coker\le\left[ \mathcal{L}_L \right]^{\vee}),
$$
which is a direct consequent of \eqref{e:psipart}.

\end{proof}

\end{subsection}

\end{section}

\begin{section}{Proofs of main theorems}\label{s:appl}
\begin{subsection}{The descent setting}\label{sub:pseudo}
Suppose that $W$ is a finitely generated torsion $\Lambda(\Gamma)$-module.
By multiplying $\psi-1$ to (\ref{e:pseu}), we obtain via the snake lemma
the exact sequence of finitely generated
$\Lambda(\Gamma')$-modules:
\begin{equation}\label{e:pseusnake}
\xymatrix{
0 \ar[r] & [W]^{\Psi} \ar[r] & W^{\Psi} \ar[r] & N^{\Psi}\ar[ld] &  &\\
 & & [W]/(\psi-1)[W] \ar[r] & W/(\psi-1)W \ar[r] &  N/(\psi-1)N \ar[r] & 0 }
\end{equation}

\begin{lemma}{\em{(Greenberg)}} \label{l:greenberg}
There exists a closed subgroup $\Gamma_0$ of $\Gamma$ mapped isomorphically onto $\Gamma'$ under the
projection $\Gamma\longrightarrow \Gamma'$ so that $N$ is
finitely generated and torsion over
$\Lambda(\Gamma_0)$.
\end{lemma}
\begin{proof} The is basically given in \cite{grn78}. We just sketch it. Since $N$ is pseudo-null, there exists an annihilator
$\phi\in\Lambda(\Gamma)$
not divided by $p$. Let ${\bar \Gamma}=\Gamma/\Gamma^p$, ${\bar \Psi}=\Psi/\Psi^p$ and let ${\bar \Gamma}_0$
be a subgroup of ${\bar \Gamma}$ mapped isomorphically onto ${\bar \Gamma}/{\bar \Psi}$ under the projection
${\bar \Gamma}\longrightarrow {\bar \Gamma}/{\bar \Psi}$. The proof of \cite[Lemma 2]{grn78} actually proves that there is a $\Gamma_0\subset \Gamma$ with
${\bar \Gamma}_0=\Gamma_0/\Gamma_0^p$ such that $N$ is finitely generated over
$\Lambda(\Gamma_0)$. And the discussion in \cite{grn78}
after the proof of Lemma 2 shows that $N$ is torsion over
$\Lambda(\Gamma_0)$.
\end{proof}
\begin{corollary}\label{c:gb}
Let $W$ be a finitely generated torsion $\Lambda(\Gamma)$-module and let $[W]$ and $N$ be as in {\em{(\ref{e:pseu})}}.
Then
\begin{enumerate}
\item[(a)] Both $N^{\Psi}$ and $N/(\psi-1)N$ are finitely generated
torsion $\Lambda(\Gamma')$-modules.
\item[(b)] We have $\chi_{\Lambda(\Gamma')}(N^{\Psi}
)=\chi_{\Lambda(\Gamma')}(N/(\psi-1)N)$.
\item[(c)] If any one of the modules $W^{\Psi}$, $[W]^{\Psi}$,
$W/(\psi-1)W$, $[W]/(\psi-1)[W]$ is torsion over $\Lambda(\Gamma')$,
then all of them are torsion. In this case, $[W]^{\Psi}=0$.
\item[(d)] 
We have $\chi_{\Lambda(\Gamma')}(W/(\psi-1)W)
=p_{L/L'}(\chi_{\Lambda(\Gamma)}(W))\cdot
\chi_{\Lambda(\Gamma')}(W^{\Psi})$.
\end{enumerate}
\end{corollary}
\begin{proof}
Identify $\Gamma'$ with the subgroup $\Gamma_0$ in Lemma \ref{l:greenberg}. Then $N$,
and hence $N^{\Psi}$ and $N/(\psi-1)N$, are
finitely generated and torsion over
$\Lambda(\Gamma_0)$. By Lemma \ref{l:coco}, (b) holds, as $\Lambda(\Gamma')=\Lambda(\Gamma_0)$.

Let $\mathfrak{Tor}(C)$ stand for the assertion that $C$ is torsion over $\Lambda(\Gamma')$.
Then the exact sequence (\ref{e:pseusnake}) implies
$$\mathfrak{Tor}([W]^{\Psi})\Longleftrightarrow \mathfrak{Tor}(W^{\Psi}),$$
$$\mathfrak{Tor}([W]/(\psi-1)[W])\Longleftrightarrow
\mathfrak{Tor}(W/(\psi-1)W).$$
Also, if $[W]=\bigoplus_{i=1}^m \Lambda(\Gamma)/\xi_i^{r_i}\Lambda(\Gamma)$ (see \S\ref{su:char}), then
$$\mathfrak{Tor}([W]/(\psi-1)[W])\Longleftrightarrow\;
(\xi_i)\not=(\psi-1),\;\; \text{for every i},\;\; \Longleftrightarrow
\mathfrak{Tor}([W]^{\Psi}),$$
and in this case, we actually have $[W]^{\Psi}=0$. Thus, (c) holds.
The assertion (d) holds trivially, if $W^{\Psi}$ is non-torsion (by (c)); otherwise, it follows from (b), (c) and
the exact sequence (\ref{e:pseusnake}), since
$p_{L/L'}(\chi_{\Lambda(\Gamma)}(W))
=\chi_{\Lambda(\Gamma')}([W]/(\psi-1)[W])$.
\end{proof}
\end{subsection}

\begin{subsection}{The proof of Theorem \ref{t:compatible}}\label{su:pft}
Now we complete the proof of Theorem \ref{t:compatible}.
\begin{proof}
Note that Proposition \ref{p:urwi}, Proposition \ref{p:go}, and Proposition \ref{p:spm} together imply that $\prod_v\vartheta_v^1=\vartheta_{L/L'}$.
Suppose $X_L$ is non-torsion over $\Lambda(\Gamma)$. We claim that $X_{L}/(\psi-1)X_L$ is non-torsion over $\Lambda(\Gamma')$.
Let $x_1,...,x_r$ be a set of generators of $X_L$ over $\Lambda(\Gamma)$. If the claim did not hold, then
there would be some $f\in\Lambda(\Gamma)$ with $p_{L/L'}(f)\not=0$ so that for each $i$ we can write
$$f\cdot x_i=\sum a_{ij}\cdot x_j,\;\; a_{ij}\in (\psi-1).$$
Write $A=\left( a_{ij}\right)$. Then all $x_i$ are annihilated by $g=\det (A-f\cdot \mathrm{I}_{r\times r})$.
As $p_{L/L'}(g)=p_{L/L'}(f)^r\not=0$, $g\not=0$ and $X_L$ is torsion, a contradiction.
Therefore, $(\Sel_{p^{\infty}}(A/L)^{\Psi})^\vee=X_{L}/(\psi-1)X_L$ is non-torsion over $\Lambda(\Gamma')$.
Now the diagram \eqref{e:diagramlong} induces the exact sequence
$$\xymatrix{ \Sel_{p^{\infty}}(A/L') \ar[r]^{\res_{L/L'}} & \Sel_{p^{\infty}}(A/L)^{\Psi} \ar[r] & \mathcal{H}^1(A,L/L')}$$
which by duality shows 
$\Theta_{L'}\vartheta_{L/L'}=0=p_{L/L'}(\Theta_L)$
as desired.

Suppose $X_L$ is torsion over $\Lambda(\Gamma)$.
If $X_{L'}$ is non-torsion over $\Lambda(\Gamma')$, then by \eqref{e:kernelres}, Corollary \ref{c:chicoh1} and Corollary \ref{c:gb}, both $X_{L}/(\psi-1)X_L$
and $[X]_{L}/(\psi-1)[X]_L$ are also non-torsion. Therefore, $\Theta_{L'}=p_{L/L'}(\Theta_L)=0$, and hence the theorem holds trivially.
If $X_{L'}$ is torsion while $\mathcal{W}^1(A,L/L')$ is non-torsion, then by Lemma \ref{l:derived} as well as its proof,
 $\mathcal{W}^2(A,L/L')$ is non-torsion and there exists some split-multiplicative place $v\in S$ such that $\vartheta_v=0$.
Therefore, $\Theta_{L'}\vartheta_{L/L'}=0$. By Proposition \ref{p:loc} and Proposition \ref{p:torcomp1},
the Pontryagin dual of $(\coker\le\left[ \mathcal{L}_L\right])^{\Psi}$ is torsion over $\Lambda(\Gamma')$. Now the diagram \eqref{e:diagramlong} yields an inclusion
$$\xymatrix{\mathcal{W}^2(A,L/L') \ar@{^{(}->}[r] & \coker\le\left[ \mathcal{L}_L^{\Psi}\right]^{\vee}} $$
that implies $\coker\le\left[ \mathcal{L}_L^{\Psi}\right]^{\vee}$ is non-torsion.
These together with \eqref{e:psipart} imply $X_L^{\Psi}$ is non-torsion. Then Corollary \ref{c:gb} implies $[X_L]/(\psi-1)[X_L]$ is non-torsion, whence $p_{L/L'}(\Theta_L)=0$.

Finally, we consider the case where $X_L$ is torsion over $\Lambda(\Gamma)$ and both $X_{L'}$ and $\mathcal{W}^1(A,L/L')$ are torsion over $\Lambda(\Gamma')$. Then by Lemma \ref{l:derived} and Corollary \ref{c:gb}(d)
$$
\begin{array}{rl}
{} & \chi_{\Lambda(\Gamma')}( \ker\le\left[
\res_{L/L'}\right]^{\vee})\cdot \chi_{\Lambda(\Gamma')}(\coker\le\left[\mathcal{L}_{L'}\right]^{\vee})
\cdot p_{L/L'}(\Theta_L)\\
= & \vartheta_{L/L'}\cdot\Theta_{L'}
\cdot
\chi_{\Lambda(\Gamma')}(\coker\le\left[ \mathcal{L}_L \right]^{\vee}/(\psi-1)\coker\le\left[ \mathcal{L}_L \right]^{\vee}).
\end{array}
$$
To proceed, we write
$$\eta_1=\chi_{\Lambda(\Gamma')}( \ker\le\left[
\res_{L/L'}\right]^{\vee})\cdot \chi_{\Lambda(\Gamma')}(\coker\le\left[\mathcal{L}_{L'}\right]^{\vee}),$$
$$\eta_2=\chi_{\Lambda(\Gamma')}(\coker\le\left[ \mathcal{L}_L \right]^{\vee}/(\psi-1)\coker\le\left[ \mathcal{L}_L \right]^{\vee}).$$
Then we deduce the desired equality $\eta_1=\varrho_{L/L'}\cdot \eta_2$, by
using Equality \eqref{e:kernelres}, Corollary \ref{c:chicoh1}, Lemma \ref{p:torcomp}, Proposition \ref{p:torcomp1} and Proposition \ref{p:loc}.
In fact, if $d\geq 3$, then $\eta_1=\eta_2=(1)$;
if $d=2$, then $\eta_1=\mathrm{w}_{L'/K}$ and $\eta_2=(1)$; if $d=1$, then $\eta_1=(|A[p^{\infty}](K)|^2/|A[p^{\infty}](K)\cap A[p^{\infty}](L)_{div}|)$ and $\eta_2=(|A[p^{\infty}](K)\cap A[p^{\infty}](L)_{div}|)$.

\end{proof}
\end{subsection}

\begin{subsection}{The proof of Theorem \ref{t:root}}\label{su:root}

Write $\Gamma^{\omega}=\ker \be\left[\e \omega\right]$ and $L^\omega=L^{\Gamma^\omega}$ for $\omega\in {\widehat{\Gamma}}$.
Denote
$$c(\omega)=\corank_{\Z_p} \Sel_{p^\infty}(A/L)^{\Gamma^\omega}.$$
\begin{lemma}\label{l:noname}
For every $\omega\in {\widehat{\Gamma}}$,
$$s(\omega)>0\;\; \Longleftrightarrow \;\; c(\omega)>c(\omega^p).$$
\end{lemma}
\begin{proof} There is an $\O_{\omega}$-homomorphism
$$\bigoplus_{\varepsilon\in {\widehat{\Gamma/\Gamma^\omega}}} (\O_{\omega}\Sel_{p^\infty})^{(\varepsilon)}\longrightarrow \O_{\omega} \Sel_{p^\infty}(A/L)^{\Gamma^\omega}$$
of finite kernel and cokernel. Since $\Gamma/\Gamma^\omega$ is cyclic, every $\varepsilon\in {\widehat{\Gamma/\Gamma^\omega}}$ equals ${}^\sigma \omega^{p^i}$
for some integer $i$ and some $\sigma\in \Gal(\overline{\Q}_p/\Q_p)$. Obviously,  $s({}^\sigma \omega^{p^i})=s(\omega^{p^i})$. Therefore,
if $|\Gamma/\Gamma^\omega|=p^n$, then
$$c(\omega)=\sum_{i=0}^n [\Q_p(\mu_{p^{n-i}}):\Q_p] \cdot s(\omega^{p^i}),$$
whence $c(\omega)=c(\omega^p)+ [\Q_p(\mu_{p^{n}}):\Q_p] \cdot s(\omega)$.
\end{proof}
Lemma \ref{l:varunramified} asserts that $\coh^1(\Gamma^{\omega}_w,A(L_v))$ is finite for every $w$ not sitting over $S$ and is trivial for almost all $w$. Since $X_L$ is finitely generated over $\Lambda(\Gamma^\omega)$, Proposition \ref{p:iwasawa} implies that
the $\Z_p$-module $\mathcal{H}^1(A,L/L^{\omega})=\prod_{w}\coh^1(\Gamma_w^{\omega},A(L_w))$, where $w$ runs over all places of $L^{\omega}$,
is cofinitely generated.
Consider the commutative diagram:
$$\xymatrix{{} &\coh^1(L^{\omega},A[p^{\infty}])_{div} \ar[r]^-{res_{\omega}}_-{\sim} \ar[d]^{\mathcal{L}_{L^{\omega}}} & (\coh^1(L,A[p^{\infty}])^{\Gamma^{\omega}})_{div} \ar[d]^{\mathcal{L}_L^{\Gamma^{\omega}}} \\
\mathcal{H}^1(A,L/L^{\omega})_{div} \ar@{^{(}->}[r] & \mathcal{H}^1(A/{L^{\omega}})_{div} \ar[r] & (\mathcal{H}^1(A/L)^{\Gamma^{\omega}})_{div}. }
$$
Here the isomorphism is due to \eqref{e:kerreslf} and \eqref{e:cokerreslf}.
The diagram and its counterpart for $\omega^p$ together
yield the commutative diagram of exact sequences (by snake lemma)
\begin{equation}\label{e:tempdiagram}
\xymatrix{\Sel_{p^{\infty}}(A/L^{\omega^p})_{div}\ar[r]^{r_{\omega^p}} \ar[d]^s & (\Sel_{p^{\infty}}(A/L)^{\Gamma^{\omega^p}})_{div} \ar[r]^{j_{\omega^p}} \ar[d]^t
& \mathcal{H}^1(A,L/L^{\omega^p})_{div} \ar[r]^{i_{L^{\omega^p}}} \ar[d]^x &
\coker\le\left[\mathcal{L}_{L^{\omega^p}}\right]_{div} \ar[d]^y \\
\Sel_{p^{\infty}}(A/L^{\omega})_{div}\ar[r]^{r_{\omega}} & (\Sel_{p^{\infty}}(A/L)^{\Gamma^{\omega}})_{div} \ar[r]^{j_{\omega}} &
\mathcal{H}^1(A,L/L^{\omega})_{div} \ar[r]^{i_{L^{\omega}}} &
\coker\le\left[\mathcal{L}_{L^{\omega}}\right]_{div},}
\end{equation}
in which all vertical arrows as well as the restriction maps $r_{\omega^p}$ and $r_{\omega}$ have finite kernels and
\begin{equation}\label{e:cap1}
\image r_{\omega}\bigcap \image t=\image r_{\omega}\circ s,
\end{equation}
\begin{equation}\label{e:cap2}
\image x \bigcap \image  j_{\omega}=\image x\circ  j_{\omega^p},
\end{equation}
\begin{equation}\label{e:cap3}
\image y \bigcap \image i_{L^{\omega}} =\image y\circ i_{L^{\omega^p}}.
\end{equation}

\begin{lemma}\label{l:temp1} The following conditions are equivalent:
\begin{enumerate}
\item[(a)] $s(\omega)> 0$.
\item[(b)] The $\Z_p$-corank of $\Sel_{p^{\infty}}(A/L^{\omega})$ is greater than that of $\Sel_{p^{\infty}}(A/L^{\omega^p})$
or the $\Z_p$-corank of $\mathcal{H}^1(A,L/L^{\omega})$ is greater than that of $\mathcal{H}^1(A,L/L^{\omega^p})$.
\end{enumerate}
\end{lemma}
\begin{proof} Lemma \ref{l:noname} says that (a) means the cokernel of $t$ has positive corank.
In view of  the equalities \eqref{e:cap1}, \eqref{e:cap2} and \eqref{e:cap3}, we only need to check that this holds if
$$\corank_{\Z_p}\coker\le\left[\mathcal{L}_{L^{\omega}}\right]> \corank_{\Z_p}\coker\le\left[\mathcal{L}_{L^{\omega^p}}\right].$$
But by \eqref{e:cokerLF}, this means $\rank_{\Z_p}\T_p\Sel(A^t/L^{\omega})>\rank_{\Z_p}\T_p\Sel(A^t/L^{\omega^p})$, or equivalently,
$$\corank_{\Z_p}\Sel_{p^{\infty}}(A/L^{\omega})>\corank_{\Z_p}\Sel_{p^{\infty}}(A/L^{\omega^p}),$$
as $A$ and $A^t$ are isogenous.
Then the desired implication follows from \eqref{e:cap1}.
\end{proof}

\begin{lemma}\label{l:h1alllambda}
The following conditions are equivalent:
\begin{enumerate}
\item[(a)]  $\corank_{\Z_p}\mathcal{H}^1(A,L/L^{\omega})
    >\corank_{\Z_p}\mathcal{H}^1(A,L/L^{\omega^p})$.
\item[(b)] There exists some split-multiplicative $v\in S$, splitting completely over $L^{\omega}/K$, such that
either $\rank_{\Z_p}\Gamma_v=1$ and $\mathfrak{w}_v=0$ or $\rank_{\Z_p}\Gamma_v\geq 2$.
\end{enumerate}
\end{lemma}
\begin{proof}
Suppose $v\in S$ is split-multiplicative.
If $v$ does not split completely over $L^{\omega}$, then the
$\coh^1(\Gamma^{\omega}_v,A(L_v))_{div}$ is fixed by $\Gamma^{\omega^p}$, hence the restriction map $\coh^1(\Gamma^{\omega^p}_v,A(L_v))_{div}\longrightarrow \coh^1(\Gamma^{\omega}_v,A(L_v))_{div}$ is surjective.
Suppose $v$ splits completely over $L^{\omega}$. Then it follows from Proposition \ref{p:spm} and Lemma \ref{l:zpsplit} 
that $v$ satisfies the condition (b) if and only
$\coh^1(\Gamma^{\omega}_v,A(L_v))_{div}$ is of positive corank,
whence $\Q_v\otimes_{\Z_p}\prod_{w\mid v}\coh^1(\Gamma_{w}^{\omega},A(L_w))^{\vee}$ contains a regular representation of $\Gamma/\Gamma^{\omega}$ over $\Q_p$.
\end{proof}

Now we prove Theorem \ref{t:root}.

\begin{proof}(of Theorem \ref{t:root})
Suppose $|\Gamma/\Gamma^{\omega}|=p^n$ for some $n$. We can
choose topological generators $\gamma_1,...,\gamma_d$ of $\Gamma$ so that $\gamma_1,\gamma_2,...,\gamma_d^{p^n}$ become
topological generators of $\Gamma^{\omega}$. For $i=1,...,d-1$, write $\Psi_i$ for the subgroup topologically generated by $\gamma_1,...,\gamma_i$
and denote $L^{(i)}=L^{\Psi_i}$, $\Gamma^{(i)}=\Gal(L^{(i)}/K)=\Gamma/\Psi_{i}$.
Then $\omega$ factors through a continuous character
$\omega^{(i)}:\Gamma^{(i)}\longrightarrow \mu_{p^n}$. As before, we have the specialization
$p_{L^{(i)}/L^{(i+1)}}:\Lambda(\Gamma^{(i)})\longrightarrow \Lambda(\Gamma^{(i+1)}),$
induced by the quotient map, as well as
$p_{\omega^{(i)}}:\O_{\omega}\Lambda(\Gamma^{(i)}) \longrightarrow\O_{\omega},$
induced by $\omega^{(i)}$. Then
$p_{L/L^{(i)}}=p_{L^{(i-1)}/L^{(i)}}\circ\dots\circ p_{L/L^{(1)}},$
and $p_{\omega}$ can be expressed as the composition
\begin{equation}\label{e:composition}
p_{\omega}=p_{\omega^{(i)}}\circ  p_{L/L^{(i)}}.
\end{equation}

Let $SP(A/L)$ denote the set of places $v\in S$ satisfying the condition (b)
of Lemma \ref{l:h1alllambda}.
Suppose $d\geq 3$.
We choose $\gamma_1,...,\gamma_d$ so that $\Psi_{d-2}$
intersects properly with both the inertia subgroup $\Gamma_v^1$ and the decomposition group $\Gamma_v$ for all $v\in S$.
Then, for $i=1,...,d-2$, the extension
$L^{(i)}/K$ remains ramified at every place $v\in S$. Furthermore, we have $$\mathrm{rank}_{\Z_p} \Gamma_v^{(i)}=
\begin{cases}
\mathrm{rank}_{\Z_p} \Gamma_v,\; &\text{if}\; \mathrm{rank}_{\Z_p} \Gamma_v\leq 2;\\
2+N,\; N\geq 0,\;& \text{otherwise.}\\
\end{cases}
$$
Then $SP(A/L)=SP(A/L^{(1)})=\cdots=SP(A/L^{(d-2)})$.
Also, by Theorem \ref{t:compatible}, for $i=1,...,d-2$,
$$\Theta_{L^{(i)}}\cdot (p^{m_i})=p_{L^{(i-1)}/L^{(i)}}(\Theta_{L^{(i-1)}}),$$
where $p^{m_i}=\prod_{v\not\in S}\vartheta_v$. These formulae together imply
$$p_{\omega}(\Theta_L)=0 \;\Longleftrightarrow \; p_{\omega^{(d-2)}}(\Theta_{L^{(d-2)}})=0.$$
It follows from Lemma \ref{l:temp1}, \ref{l:h1alllambda} that
by replacing $X_L$ by $X_{L^{(d-2)}}$, we can reduce the proof to the $d\leq 2$ case.
Suppose $d=2$ and take $L'=L^{(1)}$. If $SP(A/L)$ is non-empty containing a place $v$,
then Theorem \ref{t:compatible}(c) asserts that $p_{\omega^{(1)}}(\vartheta_v)=0$, whence
$$p_{\omega^{(1)}}(\varrho_{L/L'})\cdot p_{\omega}(\Theta_L)=p_{\omega^{(1)}}(\varrho_{L/L'}\cdot p_{L/L^{(1)}}(\Theta_L))=
p_{\omega^{(1)}}(\prod_v \vartheta_v\cdot \Theta_{L'})=0.$$
But by Lemma \ref{p:torcomp}, $p_{\omega^{(1)}}(\varrho_{L/L'})=(|A[p^\infty](L^\omega)\cap A[p^\infty](L')|)^2\not=0$.
Thus, $ p_{\omega}(\Theta_L)=0$ and the theorem holds in this case.
If $SP(A/L)$ is empty, then $SP(A/L^{(1)})$ is also empty and Theorem \ref{t:compatible} implies
$$p_{\omega}(\Theta_L)=0 \;\Longleftrightarrow \; p_{\omega^{(1)}}(\Theta_{L^{(1)}})=0.$$
Therefore, we can reduce the proof to the $d=1$ case. But, obviously, the theorem holds in the $d=1$ case.

\end{proof}

\end{subsection}

\end{section}

\end{document}